\documentclass[preprint]{amsart}
\usepackage{amsthm,amsmath,amssymb}
\usepackage{stmaryrd,mathrsfs}
\usepackage{esint}
\usepackage{lineno,hyperref}

\newcommand{\ud}[0]{\,\mathrm{d}}

\newcommand{\dist}[0]{\operatorname{dist}}

\newcommand{\abs}[1]{|#1|}
\newcommand{\Babs}[1]{\Big|#1\Big|}
\newcommand{\Norm}[2]{\|#1\|_{#2}}

\newcommand{\BNorm}[2]{\Big\|#1\Big\|_{#2}}
\newcommand{\pair}[2]{\langle #1,#2 \rangle}
\newcommand{\Bpair}[2]{\Big\langle #1,#2 \Big\rangle}
\newcommand{\ave}[1]{\langle #1\rangle}

\newcommand{\BMO}[0]{\operatorname{BMO}}
\newcommand{\supp}[0]{\operatorname{spt}}
\newcommand{\loc}[0]{\operatorname{loc}}

\newcommand{\sign}[0]{\operatorname{sgn}}
\renewcommand{\Re}[0]{\operatorname{Re}}

\newcommand{\R}{\mathbb{R}}
\newcommand{\C}{\mathbb{C}}

\newcommand{\Z}{\mathbb{Z}}

\newcommand{\Exp}[0]{\mathbb{E}}
\newcommand{\eps}[0]{\varepsilon}

\newcommand{\diam}[0]{\operatorname{diam}}

\numberwithin{equation}{section}

\theoremstyle{plain}
\newtheorem{theorem}{Theorem}[subsection]
\newtheorem{proposition}[theorem]{Proposition}

\newtheorem{lemma}[theorem]{Lemma}

\theoremstyle{definition}
\newtheorem{definition}[theorem]{Definition}

\theoremstyle{remark}
\newtheorem{remark}[theorem]{Remark}
\newtheorem{example}[theorem]{Example}

\makeatletter
\@namedef{subjclassname@2020}{%
  \textup{2020} Mathematics Subject Classification}
\makeatother

\title[Boundedness of commutators]{The $L^p$-to-$L^q$ boundedness of commutators with applications to the Jacobian operator}

\author{Tuomas P.~Hyt\"onen}
\address{Department of Mathematics and Statistics, P.O. Box~68 (Pietari Kalmin katu 5),\\ FI-00014 University of Helsinki, Finland}
\email{tuomas.hytonen@helsinki.fi}
\thanks{The author was supported by the Academy of Finland via project Nos.~314829 (Frontiers of singular integrals) and 307333 (Centre of Excellence in Analysis and Dynamics Research).}

\keywords{Commutator; singular integral; Jacobian determinant}
\subjclass[2020]{42B20; 42B25; 42B37; 35F20; 47B47}

\begin{document}

\maketitle

\begin{abstract}
Supplying the missing necessary conditions, we complete the characterisation of the $L^p\to L^q$ boundedness of commutators $[b,T]$ of pointwise multiplication and Calder\'on--Zygmund operators, for arbitrary pairs of $1<p,q<\infty$ and under minimal non-degeneracy hypotheses on $T$. 

For $p\leq q$ (and especially $p=q$), this extends a long line of results under more restrictive assumptions on $T$.
In particular, we answer a recent question of Lerner, Ombrosi, and Rivera-R{\'{\i}}os by showing that $b\in\BMO$ is necessary for the $L^p$-boundedness of $[b,T]$ for any non-zero homogeneous singular integral~$T$. We also deal with iterated commutators and weighted spaces.

For $p>q$, our results are new even for special classical operators with smooth kernels. As an application, we show that
every $f\in L^p(\R^d)$ can be represented as a convergent series of normalised Jacobians $Ju=\det\nabla u$ of $u\in \dot W^{1,dp}(\R^d)^d$.
 This extends, from $p=1$ to $p>1$, a result of Coifman, Lions, Meyer and Semmes about $J:\dot W^{1,d}(\R^d)^d\to H^1(\R^d)$, and supports a conjecture of Iwaniec about the solvability of the equation $Ju=f\in L^p(\R^d)$.

\bigskip

\noindent 
\textsc{R\'esum\'e.}
En fournissant les conditions n\'ecessaires manquantes, nous compl\'e\-tons la caract\'erisation de la bornitude $L^p \to L^q $ des commutateurs $[b, T] $ de multiplication ponctuelle et des op\'erateurs de Calder\'on--Zygmund, pour des paires arbitraires de $1<p,q<\infty $ et avec des hypoth\`eses de non-d\'eg\'en\'erescence minimales sur $T$.

Pour $p\leq q $ (et en particulier $p=q $), cela \'etend une longue ligne de r\'esultats avec des hypoth\`eses plus restrictives sur $T$.
En particulier, nous r\'epondons \`a une question r\'ecente de Lerner, Ombrosi et Rivera-R{\'{\i}}os en montrant que $b \in \BMO$ est n\'ecessaire pour la $L^p$-bornitude de $[b,T]$ pour toute int\'egrale singuli\`ere homog\`ene non nulle~$T$. Nous traitons \'egalement des commutateurs it\'er\'es et des espaces avec des poids.

Pour $p>q$, nos r\'esultats sont nouveaux m\^eme pour les op\'erateurs classiques sp\'eciaux \`a noyaux lisses. Comme application, nous montrons que
chaque $f  in L^p (\R^d) $ peut \^etre repr\'esent\'e comme une s\'erie convergente de Jacobiens normalis\'es $Ju = \det \nabla u$ de $u \in \dot W^ {1,dp} (\R^d)^d$.
 Ceci \'etend, de $p=1 $ \`a $p>1 $, un r\'esultat de Coifman, Lions, Meyer et Semmes concernant $J: \dot W^{1,d} (\R^d)^d \to H^1 (\R^d)$, et supporte une conjecture d'Iwaniec sur la r\'esolvabilit\'e de l'\'equation $Ju = f \in L^p (\R^d) $.
 \end{abstract}

\section{Introduction}

The first goal of this paper is to complete the following picture of the $L^p(\R^d)$-to-$L^q(\R^d)$ boundedness properties of commutators of pointwise multiplication and singular integral operators:

\begin{theorem}\label{thm:summary}
Let $1<p,q<\infty$, let $T$ be a ``non-degenerate'' Calder\'on--Zygmund operator on $\R^d$, and let $b\in L^1_{\loc}(\R^d)$. Then the commutator
\begin{equation*}
  [b,T]:f\mapsto bTf-T(bf)
\end{equation*}
defines a bounded operator $[b,T]:L^p(\R^d)\to L^q(\R^d)$ if and only if:
\begin{itemize}
  \item $p=q$ and $b$ has bounded mean oscillation, or 
  \item $\displaystyle p<q\leq p^* =\frac{pd}{(d-p)_+}$ and $b$ is $\alpha$-H\"older continuous for $\displaystyle \alpha=\Big(\frac{1}{p}-\frac{1}{q}\Big)d$, or
  \item $q>p^*$ and $b$ is constant, or
  \item $p>q$ and $b=a+c$, where $a\in L^r(\R^d)$ for $\displaystyle \frac1r=\frac1q-\frac1p$, and $c$ is constant.
\end{itemize}
\end{theorem}

To be explicit, the definition of the Sobolev exponent $p^*$ above is $pd/(d-p)$, if $p<d$, and $\infty$ otherwise;
thus $p<q\leq p^*$ is precisely the condition that the H\"older exponent satisfies $\alpha\in(0,1]$. We say that a Calder\'on--Zygmund operator $Tf(x)=\int K(x,y)f(y)\ud y$, with usual (or weaker) assumptions on the kernel $K$ recalled in Section \ref{ss:non-deg}, 
is ``non-degenerate'' provided that, for some $c_0>0$,
\begin{equation}\label{eq:nondeg}
   \text{for every }y\in\R^d\text{ and }r>0,\text{ there is }x\in B(y,r)^c\text{ with }\abs{K(x,y)}\geq\frac{1}{c_0 r^d};
\end{equation}
i.e., uniformly over all positions and length-scales, the kernel takes some values that are as big as they are allowed to be by the standard upper bound for $K(x,y)$. When $\displaystyle K(x,y)=\frac{\Omega(x-y)}{\abs{x-y}^d}$ is a (possibly rough) homogeneous kernel, this requirement simply says that $\Omega$ is not identically zero.

\subsection{Sufficient conditions for boundedness}\label{ss:intro-suff}
We note that all the ``if'' parts of Theorem \ref{thm:summary} are either well known or easy. The cases when $b$ is constant are completely trivial, since in this case the commutator vanishes. 
If $b\in L^r(\R^d)$ with $\frac1r=\frac1q-\frac1p$, the boundedness is also immediate simply from the boundedness of $T$ on both $L^p(\R^d)$ and $L^q(\R^d)$ (taking this as part of the definition of a ``Calder\'on--Zygmund operator''), together with H\"older's inequality:
\begin{equation*}
\begin{split}
  \Norm{[b,T]f}{q}
  \leq\Norm{bTf}{q}+\Norm{T(bf)}{q}
  &\leq\Norm{b}{r}\Norm{Tf}{p}+\Norm{T}{L^q\to L^q}\Norm{bf}{q} \\
  &\leq\Norm{b}{r}\big(\Norm{T}{L^p\to L^p}+\Norm{T}{L^q\to L^q}\big)\Norm{f}{p}.
\end{split}
\end{equation*}
In particular, no mutual cancellation between the two terms of the commutator is involved in this estimate.
This computation is also valid when $p=q$ and $r=\infty$, showing the trivial sufficiency of $b\in L^\infty(\R^d)$ for the boundedness of $[b,T]$ on $L^p(\R^d)$. The fact that the larger space $\BMO(\R^d)$ is still admissible for this boundedness is a celebrated theorem of Coifman, Rochberg and Weiss \cite{CRW} and the only truly nontrivial result among the ``if'' statements of Theorem \ref{thm:summary}.

If $b$ is $\alpha$-H\"older continuous, using only the standard pointwise bound for Calder\'on--Zygmund kernels, we see that
\begin{equation*}
  \abs{[b,T]f(x)}=\Babs{\int_{\R^d}(b(x)-b(y))K(x,y)f(y)\ud y}
  \lesssim\int_{\R^d}\abs{x-y}^\alpha\frac{1}{\abs{x-y}^d}\abs{f(y)}\ud y
\end{equation*}
is pointwise dominated by the usual fractional integral operator, whose $L^p(\R^d)$-to-$L^q(\R^d)$ bounds are classical and well known.

\subsection{Necessary conditions for boundedness}\label{ss:intro-nec}
Let us then discuss the ``only if'' parts of Theorem \ref{thm:summary}. For $p=q$, already Coifman, Rochberg and Weiss \cite{CRW} proved the necessity of $b\in\BMO(\R^d)$ for the $L^p(\R^d)$-boundedness of $[b,T]$ for {\em all} $d$ Riesz transforms $R_j$, $j=1,\ldots,d$. (This reduces to just the Hilbert transform when $d=1$.) Their argument made explicit use of the special algebraic form of the relevant kernels.

Janson \cite{Janson:1978} and Uchiyama \cite{Uchiyama:1978}, independently, extended the necessity part of the Coifman--Rochberg--Weiss theorem to more general classes of homogeneous Calder\'on--Zygmund kernels with ``sufficient'' smoothness. In particular, their results contain the fact that the boundedness of $[b,R_j]$ for just {\em one} (instead of {\em all}) $j=1,\ldots,d$ already implies that $b\in\BMO(\R^d)$.
Janson's argument may be viewed as an analytic extension of that of Coifman et al., in that he used the smoothness to guarantee absolute convergence of the Fourier expansion of the inverse $1/K$ of the kernel, where the individual frequency components could then be treated by the algebraic method. Janson also proves the ``only if'' part of Theorem \ref{thm:summary} for $p<q$ (and in fact for more general Orlicz norms) for the same class of smooth homogeneous kernels. Uchiyama's argument is different, but still dependent on both smoothness and homogeneity of the kernel.

A recent advance was made by Lerner, Ombrosi and Rivera-R{\'{\i}}os \cite{LORR}, who identified sufficient {\em local positivity} (lack of sign change in a nonempty open set) as a workable replacement of the previous smoothness assumptions on the (still homogeneous) kernel to deduce the necessity of $b\in\BMO(\R^d)$ for the $L^p(\R^d)$-boundedness of $[b,T]$. 
Similar results in the case of not necessarily homogeneous Calder\'on--Zygmund kernels were subsequently obtained by Guo, Lian and Wu \cite{GLW}; see also Duong, Li, Li and Wick \cite{DLLW} for the concrete case when $T$ is a Riesz transform related to the sub-Laplacian on a stratified nilpotent Lie group.

In the present work, we take the final step in generalising the class of admissible kernels, showing that any non-degenerate Calder\'on--Zygmund kernel is admissible for the ``only if'' conclusions of 
Theorem \ref{thm:summary}. In particular, our result applies to both two-variable kernels $K(x,y)$ (with very little smoothness) and rough homogeneous kernels $\displaystyle \frac{\Omega(x-y)}{\abs{x-y}^d}$, under a minimal non-degeneracy assumption. In the case of homogeneous kernels we merely need that $\Omega\in L^1(S^{d-1})$ does not vanish identically. This answers positively a question raised by Lerner et al. \cite[Remark 4.1]{LORR}; as discussed below, we also address the more general two-weight bounds and higher commutators as considered in \cite{LORR}. Also in the case of two-variable kernels, our non-degeneracy hypothesis seems to be at least as general as anything found in the literature; in contrast to \cite{GLW} in particular, we allow in \eqref{eq:nondeg} that the point of non-degeneracy $x$ may lie in any direction from the reference point~$y$.

\subsection{The case $p>q$ and applications to the Jacobian operator}\label{ss:intro-Jacobi}
The case $p>q$ of Theorem \ref{thm:summary} is completely new even for special Calder\'on--Zygmund operators like the Riesz transforms, for which the complementary range $p\leq q$ was understood for a long time. The result in this new range is perhaps surprising, in that it says that there is essentially no cancellation between $bT$ and $Tb$ in this regime. (An initial working hypothesis before discovering this result was that the role of BMO in the commutator boundedness in this regime of exponents could be taken by another space $JN_r$, which was implicitly introduced by John and Nirenberg \cite[\S 3]{JN:61} and recently studied in \cite{DHKY}. However, the obtained result disproves this hypothesis.)

Technically, this is the hardest case of the proof, which is somewhat explained by the fact that membership in $L^r(\R^d)$ is a ``global'' condition, in contrast to the ``uniform local'' conditions defining both $\BMO(\R^d)$ and $\alpha$-H\"older continuous functions. Incidentally, a similar dichotomy between ``global'' conditions characterising $L^p$-to-$L^q$ (or similar) boundedness for $p>q$, and ``uniform local'' conditions in the case $p\leq q$, has also been recently discovered in a couple of other settings as well:
\begin{enumerate}
  \item In the context of two-weight norm inequalities for certain discrete positive operators, the characterisation for $p\leq q$ by Lacey, Sawyer and Uriarte-Tuero \cite{LSU:positive} is in terms of local ``testing conditions'' uniform over all dyadic cubes, while the characterisation for $p> q$ due to Tanaka \cite{Tanaka:Wolff} involves the $L^r$ membership of a ``discrete Wolff potential''; see also \cite{HHL} for a unified approach to both cases. 
  \item The boundedness of certain Toeplitz type operators between the holomorphic Hardy spaces $H^p$ and $H^q$ of the unit ball was characterised by Pau and Per\"al\"a \cite{PauP} in both regimes of the exponents, in terms of a uniform local Carleson measure condition for $p\leq q$, and in terms of the global $L^r$ membership of a certain auxiliary function for $p>q$. These results, analogous to our present ones but in a different context, were found independently at almost the same time: the first arXiv versions of \cite{PauP} and the present work came out within two weeks of each other.
\end{enumerate}
It might be of interest for general operator theory in $L^p$ spaces to find further examples of, and/or a broader context for, this phenomenon.

A part of the motivation to study this regime of exponents for commutator inequalities came from a recent observation of Lindberg \cite{Lindberg:2017} about the connections of such bounds, in the particular case when $T$ is the Ahlfors--Beurling transform, to the Jacobian equation
\begin{equation*}
  Ju:=\det\nabla u:=\det(\partial_i u_j)_{i,j=1}^d= f\in L^p(\R^d).
\end{equation*}
It has been conjectured by Iwaniec \cite{Iwaniec:Escorial} that, for $p\in(1,\infty)$, the (obviously bounded) map $J:\dot W^{1,pd}(\R^d)^d\to L^p(\R^d)$, where $\dot W^{1,pd}$ is the homogeneous Sobolev space,  has a continuous right inverse and in particular is surjective. As a variant of our estimates for commutators, we will provide partial positive evidence by showing that the closed linear span of the range of $J$ is all of $L^p(\R^d)$. This is an $L^p$-analogue of a result of Coifman, Lions, Meyer and Semmes \cite[p.~258]{CLMS} who obtained a similar conclusion for $J:\dot W^{1,d}(\R^d)^d\to H^1(\R^d)$, which corresponds to the case $p=1$, with the usual replacement of $L^1$ by the Hardy space $H^1$. 

Recently, Lindberg \cite[p. 739]{Lindberg:2017} proposed an approach to the planar ($d=2$) case of the Jacobian operator via the complex-variable framework
\begin{equation*}
  Ju=\abs{\partial h}^2-\abs{\bar\partial h}^2
    =\abs{S(\bar\partial h)}^2-\abs{\bar\partial h}^2,
  \end{equation*}
where $h=u_1+iu_2$, $\partial=\frac12(\partial_1-i\partial_2)$, $\bar\partial=\frac12(\partial_1+i\partial_2)$,
and $S$ is the Ahlfors--Beurling operator. This led him to a question about the boundedness of the commutator $[b,S]:L^{2p}\to L^{(2p)'}$, which is solved as a particular case of Theorem \ref{thm:summary}; observe that $2p>2>(2p)'$ here. Following Lindberg's outline \cite[p. 739]{Lindberg:2017}, conclusions about the planar Jacobian could then be obtained as corollaries to Theorem \ref{thm:summary}; but it turns out that a combination of some elements of its proof, together with the techniques of Coifman, Lions, Meyer and Semmes \cite{CLMS}, actually allows to prove such results in any dimension; see Section \ref{sec:Jacobi}.

\subsection{A priori assumptions on $b$, $T$ and $[b,T]$}
In general it takes some effort to define precisely what is meant by ``$Tf$'', when $T$ is a singular integral operator, or by saying that such an operator ``is bounded'' from one space to another.
In our approach to the ``only if'' statements of Theorem \ref{thm:summary}, we avoid all this subtlety; in fact, our assumptions may be formulated entirely in terms of the kernel $K$ without ever having to define the operator $T$ or $[b,T]$, although we still use these symbols as convenient abbreviations. All we need is estimates for the bilinear form
\begin{equation}\label{eq:bilin}
  \pair{[b,T]f}{g}=\iint\big(b(x)-b(y))K(x,y)f(y)g(x)\ud y\ud x,
\end{equation}
where the functions $f,g\in L^\infty(\R^d)$ have bounded supports separated by a positive distance; we refer to such estimates as {\em off-support} bounds for $[b,T]$.
Under the standard estimates for a Calder\'on--Zygmund kernel, the above integral exists as an absolutely convergent Lebesgue integral when $b\in L^1_{\loc}(\R^d)$, as in Theorem \ref{thm:summary}.

For $p\leq q$, we only need the bound
\begin{equation}\label{eq:restrOffSuppWeakType}
\begin{split}
  \abs{\pair{[b,T]f}{g}} &\leq  C\Norm{f}{\infty}\Norm{g}{\infty}\abs{B}^{1/p}\abs{\tilde B}^{1/p'},\quad\text{whenever}\\
  &\quad \supp f\subset B,\quad \supp g\subset \tilde B,\quad r_B=r_{\tilde B}\leq c\dist(B,\tilde B),
\end{split}
\end{equation}
where $B$ and $\tilde B$ denote arbitrary balls of radius $r_B$ and $r_{\tilde B}$, respectively.
This is weaker than a {\em restricted weak type $(p,q)$ estimate} in two ways: the bound involves the bigger quantities $\abs{B}$ in place of $\abs{\tilde B}$ on the right, and it is only required to hold under the quantitative off-support condition above. (A certain technical strengthening, but still formally weaker than the global boundedness of $[b,T]:L^p\to L^q$, and involving off-support bounds only, is needed when $p>q$.)

We note that Liaw and Treil \cite{LiawTreil} have provided a framework to interpret the boundedness of a singular integral operator (an issue that we have chosen to avoid) via off-support conditions of a similar flavour. However, the off-support conditions that we impose on $f$ and $g$ are significantly stronger (and hence the resulting estimate on the operator restricted to such pairs of functions much weaker) than those of \cite{LiawTreil}; in particular, the quantitative separation of supports in \eqref{eq:restrOffSuppWeakType} efficiently prevents approximating a form with arbitrary $f,g$ (as done in \cite{LiawTreil}) by the off-support forms above.

The fact that one only needs off-support estimates in the ``only if'' directions of Theorem \ref{thm:summary} is already implicit in the argument of Uchiyama \cite[proof of Theorem~1]{Uchiyama:1978}, but not in all recent works, and it seems not to have been explicitly stated in the literature. On the other hand, Lerner et al. \cite{LORR} use a {\em restricted strong type} assumption, while Guo et al. \cite{GLW} state one of their results under a weak type hypothesis. Our condition \eqref{eq:restrOffSuppWeakType} simultaneously relaxes both these assumptions.

Note that the a priori assumption that $b\in L^1_{\loc}(\R^d)$ is essentially the weakest possible to make sense of the commutator $[b,T]$, even in the off-support sense as above. While many earlier results related to Theorem \ref{thm:summary} are obtained under this same minimal assumption, some others assume $b\in\BMO(\R^d)$ {\em qualitatively} to begin with, and then prove the {\em quantitative} bound $\Norm{b}{\BMO}\lesssim\Norm{[b,T]}{L^p\to L^p}$; see e.g. \cite[Theorem 1.2]{DLLW}. A simplification brought by this stronger a priori assumption is that one can absorb error terms of the form $\eps\Norm{b}{\BMO}$ in the argument. We will also use absorption, but only to quantities whose finiteness is guaranteed by $b\in L^1_{\loc}(\R^d)$. 

\subsection{Methods and scope}
We will prove versions of Theorem \ref{thm:summary} by two methods of somewhat different scopes. The first method is based on the well-known connection of commutator estimates to {\em weak factorisation}, which has been widely used since the pioneering work \cite{CRW}. (In contrast to proper factorisation, where an object is expressed as a product of other objects, weak factorisation refers to decompositions in terms of sums, or possibly infinite series, of products.) This depends on the basic identity
\begin{equation*}
  \pair{[b,T]f}{g}=\pair{b}{gTf-fT^*g},
\end{equation*}
where each term is well-defined as a Lebesgue integral for disjointly supported $f$ and $g$. Hence, if an arbitrary $h$ in (a dense subspace of) a predual of the space hoped to contain $b$ can be expanded as
\begin{equation}\label{eq:weakFact}
  h=\sum_i (g_i Tf_i-f_i T^*g_i),
\end{equation}
then we can hope to estimate
\begin{equation*}
  \abs{\pair{b}{h}}
  \leq\sum_i\abs{\pair{b}{g_i Tf_i-f_i T^*g_i}}
  =\sum_i\abs{\pair{[b,T]f_i}{g_i}}
\end{equation*}
in order to bound $\Norm{b}{}$ in terms of $\Norm{[b,T]}{}$. An inherent difficulty is that, even with good convergence properties of the expansion \eqref{eq:weakFact} in the predual space, lacking the a priori knowledge that $b$ should be in the relevant space, it may be difficult to justify the ``$\leq$'' above. We circumvent this problem by replacing \eqref{eq:weakFact} by an {\em approximate weak factorisation}, where the sum over $i$ is finite, but there is an additional error term $\tilde h$ that will be eventually absorbed.

This method is strong enough for proving Theorem \ref{thm:summary} as stated, where both the function $b$ and the kernel $K(x,y)$ of $T$ are allowed to be {\em complex-valued}. Besides completeness of the theory, achieving this level of generality was initially motivated by the applications to the Jacobian operator via the Ahlfors--Beurling transform, as discussed above. The kernel of this operator, $K(z,w)=-\pi^{-1}/(z-w)^2$ for $z,w\in\C$, is genuinely complex-valued, and it is only natural to view it as acting on (and forming commutators with) complex-valued functions. While this is hardly exotic, it should be stressed that some of the recent contributions, like our second method, are inherently restricted to real-valued~$b$.

Our second approach could be called the {\em median method}, and it is a close cousin of the recent work \cite{LORR}. It makes explicit use of the order structure of the real line as the range of the function $b$. The advantage of this method is that, with little additional effort, it can also handle the higher order commutators
\begin{equation*}
  T_b^k=[b,T_b^{k-1}],\qquad T_b^1=[b,T].
\end{equation*}
As before, we only need the off-support bilinear form
\begin{equation*}
  \pair{T_b^k f}{g}=\iint (b(x)-b(y))^k K(x,y)f(y)g(x)\ud u\ud x
\end{equation*}
of these operators for $f,g\in L^\infty(\R^d)$ with bounded supports separated by a positive distance, and $b\in L^k_{\loc}(\R^d)$ is a sufficient a priori assumption to make sense of this.
We also apply this method to two-weight commutator inequalities in Section \ref{ss:two-weight}.

\subsection{Extensions to other settings}

While the present work concentrates on commutators $[b,T]$ (and their iterates of the form $[b,[b,T]]$ etc.)~of pointwise multipliers and linear Calder\'on--Zygmund operators between $L^p$ spaces, we record a number of extensions of either the operators or the spaces under consideration:
\begin{enumerate}
  \item Bilinear Calder\'on--Zygmund operators map $T:L^{p_1}\times L^{p_2}\to L^p$ with $\frac1p=\frac{1}{p_1}+\frac{1}{p_2}$,
and one can ask about conditions for
\begin{equation*}
  [b,T]_1:(f,g)\mapsto bT(f,g)-T(bf,g):L^{p_1}\times L^{p_2}\to L^q.
\end{equation*}
The necessity of $b\in\BMO$ when $p=q$ was first obtained by Chaffee~\cite{Chaffee:16} under Janson-type assumptions and methods involving the Fourier expansion of the inverse kernel $1/K$. Since the circulation of our results, Oikari~\cite{Oikari:awf} has extended the present hypotheses and methods to bilinear operators, obtaining a close analogue of Theorem \ref{thm:summary} in this setting.

  \item Iterated commutators of the form $[[b,T_1],T_2]$, where $b$ is a function of two variable $x_1,x_2$, and each $T_i$ is a Calder\'on--Zygmund operator acting in the variable $x_i$, play an important role in the theory of singular integrals and function spaces on product domains. A characterisation of the $L^p\to L^p$ boundedness of these bi-commutators remains open, after the recent discovery of a gap~\cite{Lacey:error} in the celebrated Ferguson--Lacey theorem \cite{FL:Acta} and its extensions; nevertheless, various mixed-norm $L^{p_1}(L^{p_2})\to L^{q_1}(L^{q_2})$ bounds with $(p_1,p_2)\neq(q_1,q_2)$ have been recently characterised by Airta et al. \cite{AHLMO}, by extending the methods of the present paper.
  
  \item The necessity of $b\in\BMO(\R^d)$ for the boundedness of commutators of both linear and bilinear singular integrals between more general Banach functions spaces (in place of $L^p$ spaces) has been obtained  by Chaffee and Cruz-Uribe \cite{ChaffeeCruz}, again under Janson-type assumptions and approach. It might be of interest to revisit their results with our (more general) conditions and methods. Among other examples, Chaffee and Cruz-Uribe also consider weighted bounds, in which case our Theorem \ref{thm:Bloom2} is a significant generalisation of their \cite[Corollary 2.3]{ChaffeeCruz}. It seems plausible that similar extensions would be available for other results of \cite{ChaffeeCruz} as well.
\end{enumerate}

\subsection{About notation}\label{ss:notation}
We will make extensive use of the notation ``$\lesssim$'' to indicate an inequality up to an unspecified multiplicative constant. Such constants are always allowed to depend on the underlying dimension $d$, any of the Lebesgue space exponents $p,q,r,\ldots$, and also on the Calder\'on--Zygmund operator $T$ and its kernel $K$, as well as on the order $k$ of an iterated commutator; these are regarded as fixed throughout the argument. The implied constants may never depend on any of the functions under consideration (neither on the function $b$ appearing in the commutator $[b,T]$ itself nor on any of the functions $f,g,\ldots$ on which the commutator acts), nor points or subsets (balls, cubes, etc.) of their domain $\R^d$. Many arguments involve an auxiliary (large) parameter $A$, and dependence on it is also indicated explicitly until a suitable value of $A$ (depending only on the admissible quantities) is fixed once and for all for the rest of the argument.

The subscript zero of a Lebesgue space indicates vanishing integral, i.e., $L^p_0(Q)=\{f\in L^p(Q):\int f=0\}$. The subscript zero of a Sobolev space $W^{1,p}_0(\Omega)$ (which will be only mentioned in passing) indicates vanishing boundary values in the Sobolev sense. Compact support is indicated by the subscript $c$, mainly in the context of the test function space $C_c^\infty(\R^d)$. We denote by $\fint_E f:=\abs{E}^{-1}\int_E f$ the average of a function over a set $E$ of finite positive measure.

\section{Complex commutators and approximate weak factorisation}\label{sec:complex}

In this section we prove the ``only if'' claims of Theorem \ref{thm:summary}.

\subsection{Non-degenerate Calder\'on--Zygmund kernels}\label{ss:non-deg}

We begin by describing the precise class of singular integral kernels that we study.
We consider two-variable Calder\'on--Zygmund kernels under the standard conditions
\begin{equation*}
  K(x,y)\leq\frac{c_K}{\abs{x-y}^d}\qquad\forall x\neq y,
\end{equation*}
\begin{equation*}
  \abs{K(x,y)-K(x',y)}+\abs{K(y,x)-K(y,x')}
  \leq \frac{1}{\abs{x-y}^d}\omega\Big(\frac{\abs{x-x'}}{\abs{x-y}}\Big),
\end{equation*}
whenever $\abs{x-x'}<\frac12\abs{x-y}$, where the modulus of continuity $\omega:[0,1)\to[0,\infty)$ is increasing.
We refer to such a kernel as an {\em $\omega$-Calder\'on--Zygmund kernel}.
A common assumption is that $\omega(t)=c_\alpha t^\alpha$ for some $\alpha\in(0,1]$, or a more general Dini-condition $\int_0^1\omega(t)\frac{\ud t}{t}<\infty$, but we need even significantly less, namely that $\omega(t)\to 0$ as $t\to 0$.

We also consider {\em rough homogeneous kernels}
\begin{equation*}
  K(x,y)=K(x-y)=\frac{\Omega(x-y)}{\abs{x-y}^d},
\end{equation*}
where $\Omega\in L^1(S^{d-1})$ and $\Omega(tx)=\Omega(x)$ for all $t>0$ and $x\in\R^d$. We note that the off-support bilinear form \eqref{eq:bilin} is also well defined (absolutely integrable) for this type of kernels: the integrals of $y\mapsto\abs{K(x-y)f(y)}$ are uniformly bounded over $x\in\supp g$, and $x\mapsto\abs{b(x)g(x)}$ is integrable; the term involving $b(y)$ can be estimated similarly by carrying the iterated integrals in a different order.

In either case, the $L^p(\R^d)$-boundedness of an integral operator $T$ associated with $K$ neither follows from these assumptions, nor is assumed as a separate condition, as this is not needed. The story is different for the ``if'' directions of Theorem \ref{thm:summary}, but our present goal is to prove the ``only if'' directions with minimal assumptions.

\begin{definition}\label{def:ndCZK}
We say that $K$ is a non-degenerate Calder\'on--Zygmund kernel, if (at least) one of the following two conditions holds:
\begin{enumerate}
  \item\label{it:ndVar} $K$ is an $\omega$-Calder\'on--Zygmund kernel with $\omega(t)\to 0$ as $t\to 0$ and for every $y\in\R^d$ and $r>0$, there exists $x\in B(y,r)^c$ with
\begin{equation*}
  \abs{K(x,y)}\geq \frac{1}{c_0 r^{d}}.
\end{equation*}
  \item\label{it:ndRough} $K$ is a homogeneous Calder\'on--Zygmund kernel with $\Omega\in L^1(S^{d-1})\setminus\{0\}$. In particular, there exists a Lebesgue point $\theta_0\in S^{d-1}$ of $\Omega$ such that $$\Omega(\theta_0)\neq 0.$$
\end{enumerate}
\end{definition}

\begin{remark}[Comparison with non-degenerate kernels in the sense of Stein]\label{rem:ndCZK}
Suppose that $K$ is an $\omega$-Calder\'on--Zygmund kernel of the convolution form $K(x,y)=K(x-y)$. Then the non-degeneracy condition \eqref{it:ndVar} of Definition \ref{def:ndCZK} simplifies into the following form: for every $r>0$, we have
\begin{equation}\label{eq:ndOur}
  \abs{K(x)}\geq\frac{1}{c_0 r^d}\quad\text{for some }x\in B(0,r)^c.
\end{equation}
For convolution kernels, there is also the following well-known non-degeneracy condition introduced by Stein \cite[IV.4.6]{Stein:book}: there exists a constant $a>0$, and a unit vector $u_0$, so that
\begin{equation}\label{eq:ndStein}
  \abs{K(t\cdot u_0)}\geq a\cdot \abs{t}^{-d},\quad\text{for all }t\in\R\setminus\{0\}.
\end{equation}
It is immediate that Stein's non-degeneracy implies our version. In fact, assume \eqref{eq:ndStein} and fix some $c_1\geq 1$. Given $r>0$, we find that any $x=t\cdot u_0$, where $\abs{t}\in[r,c_1 r]$, satisfies\eqref{eq:ndOur} with $c_0=c_1^d/a$. Thus, while \eqref{eq:ndOur} requires just the existence of one $x$, Stein's condition provides two symmetric line segments of admissible $x$ that, moreover, have simple explicit dependence on $r$ and are always located on the same fixed ray through the origin. It is not surprising that \eqref{eq:ndOur} is easily satisfied even when \eqref{eq:ndStein} is not, and we provide some examples below.

Note that it is assumed in the discussion of non-degeneracy in \cite[IV.4.6]{Stein:book}, but not  in Definition \ref{def:ndCZK}, that $K$ should be the kernel of a bounded operator on $L^2(\R^d)$, and this would offer a source of cheap examples in terms of kernels of unbounded operators. To make clear that this is not a decisive difference between the two conditions, we take the slight additional trouble of making our examples correspond to bounded operators on $L^2(\R^d)$.
\end{remark}

Several of our examples to follow will exploit a standard resolution of unity
\begin{equation}\label{eq:reso}
  \sum_{j\in\Z}\varphi(2^{-j}x)\equiv 1\quad\forall x>0,\quad\varphi\in C_c^\infty(\frac12,2),
\end{equation}
where we note in particular that $\varphi(1)=1$ under these conditions.

\begin{example}[Stein's non-degeneracy violated at one or two points]\label{ex:1point}
When $d=1$, Stein's condition \eqref{eq:ndStein} simply says that $\abs{K(x)}\geq a\abs{x}^{-1}$, so any $K$ that vanishes even at one point of $\R\setminus\{0\}$ is not admissible. Let us fix some $K_0$ that does satisfy \eqref{eq:ndStein}, say the Hilbert kernel $K_0(x)=1/x$. We then define
\begin{equation*}
  K(x)=K_0(x)-K_0(1)\varphi(x).
\end{equation*}
It is immediate that this perturbation of $K_0$ neither destroys the Calder\'on--Zygmund kernel bounds nor the $L^2(\R)$-boundedness of the operator. But $K(1)=0$, so \eqref{eq:ndStein} is clearly violated. In contrast, \eqref{eq:ndOur} trivially holds; we can e.g. take $x=-r$ for any given $r>0$. If we also subtract $K_0(-1)\varphi(-x)$, so as to violate Stein's condition at both $x=\pm 1$, we still have \eqref{eq:ndOur}, where we can e.g. take $x=\pm r$ when $r\in(0,\frac12]\cup[2,\infty)$ and $x=\pm 2$ when $r\in(\frac12,2)$.
\end{example}

\begin{example}[Stein's non-degeneracy violated in a half-space, $d\geq 2$]\label{ex:halfspace}
Let $d\geq 2$ and consider a homogeneous convolution kernel
\begin{equation*}
  K(x)=\frac{x_i x_j}{\abs{x}^{d+2}}\cdot 1_{(0,\infty)}(x_j)
    =\frac{1}{\abs{x}^d} \Omega\big(\frac{x}{\abs{x}}\big),\quad\Omega(x)=x_i\cdot x_j\cdot 1_{(0,\infty)}(x_j);
\end{equation*}
this is the truncation of a second order Riesz transform to a half space. Since $s\mapsto s\cdot 1_{(0,\infty)}(s)$ is Lipschitz-continuous, $K(x-y)$ is an $\omega$-Calder\'on--Zygmund kernel with $\omega(t)=t$, and it also satisfies $\int_{S^{d-1}}\Omega(u)\ud\sigma(u)=0$. Under these conditions, it is classical that $K$ is the convolution kernel of a bounded operator on $L^2(\R^d)$ (see e.g. \cite[Proposition II.5.5]{GCRF}). For any unit vector $u_0$, it is clear that $K(t\cdot u_0)$ must vanish for either all $t\in(0,\infty)$ or all $t\in(-\infty,0)$, so that Stein's condition \eqref{eq:ndStein} is impossible. On the other hand, for any $r>0$, choosing $x=2^{-1/2}(e_i+e_j)r$, where $e_i,e_j$ are standard unit vectors, we have $x\in B(0,r)^c$ and $K(x)=(2^{-1/2}r)^2/r^{d+2}=2^{-1}r^{-d}$, so that $K$ satisfied Definition \ref{def:ndCZK}\eqref{it:ndVar}.
\end{example}

\begin{example}[Stein's non-degeneracy violated on a half-line, $d=1$]\label{ex:onesided}
In dimension $d=1$, it takes a bit more effort to construct an analogue of Example \ref{ex:halfspace}; but this pays off, as it allows us to connect the example to the theory of {\em one-sided singular integrals} introduced by Aimar, Forzani and Mart\'in-Reyes~\cite{AFMR}. These are simply convolution-type $\omega$-Calder\'on--Zygmund kernel $K$ supported on $(0,\infty)$. The basic example of a non-trivial one-sided kernel provided in \cite[(1.5)]{AFMR}, $K(x)=1_{(0,\infty)}(x)\cdot x^{-1}\cdot\sin(\log x)/\log x$, decays a bit too fast at $0$ and $\infty$ to satisfy Definition \ref{def:ndCZK}\eqref{it:ndVar}, but a non-degenerate example can be given as follows: With the resolution of unity \eqref{eq:reso}, let
\begin{equation*}
  K(x):=\sum_{j\in\Z}\varphi(2^{-j}x)\frac{(-1)^j}{x}.
\end{equation*}

It is immediate that this satisfies the higher order Calder\'on--Zygmund estimates $\abs{x}^n\abs{D^n k(x)}\leq c_n$ for all $n=0,1,2,\ldots$, and in particular the $\omega$-Calder\'on--Zygmund estimates with $\omega(t)=t$. Since consecutive bumps in the series of $K$ have equal integral with opposite signs, $K$ also satisfies the usual cancellation condition $\abs{\int_{\eps<\abs{x}<N}K(x)\ud x}\leq C$ for all $0<\eps<N<\infty$. However, it fails to satisfy the existence of the limit $\lim_{\eps\to 0}\int_{\eps<\abs{x}<1}K(x)\ud x$, which is needed to define the associated principal value convolution operator in the classical theory. But if we take the limit $\eps\to 0$ only along the powers $\eps=4^{-n}$, $n\in\Z$ (so as to proceed in steps of two consecutive bumps of opposite signs), then the relevant limit exists, and a trivial modification of the standard theory (see e.g. \cite[Proposition II.5.5]{GCRF}) shows that $Tf(x):=\lim_{n\to\infty}\int_{\abs{x-y}>4^{-n}}K(x-y)f(y)\ud y$ defines a bounded operator on $L^2(\R^d)$ with convolution kernel $K$. Finally, this $K$ easily satisfies Definition \ref{def:ndCZK}\eqref{it:ndVar}: Given $r>0$, let $r\leq 2^j<2r$ so that $x=2^j\in B(0,r)^c$ satisfies $\abs{K(x)}=\abs{(-1)^j x^{-1}}\geq(2r)^{-1}$.

Recall that Stein's non-degeneracy condition was introduced for the following result \cite[IV.4.6, Proposition 7]{Stein:book}: If a convolution operator with non-degenerate kernel in Stein's sense acts boundedly on a weighted space $L^p(w)$, then the weight $w$ must belong to Muckenhoupt's class $A_p$. On the other hand, Aimar et al. \cite{AFMR} show that their one-sided operators, and hence in particular the example that we just gave, act boundedly on $L^p(w)$ for a strictly larger weight class $A_p^-$. As we will show in this paper, non-degeneracy in the sense of Definition \ref{def:ndCZK}\eqref{it:ndVar} (which is satisfied by the said example) is enough to imply various necessary conditions on $b$ for the boundedness of the commutator $[b,T]$. In particular, a weaker notion of non-degeneracy of a singular integral $T$ is needed to deduce that $b\in\BMO$ from the $L^p$-boundedness of $[b,T]$, than what is needed to deduce that $w\in A_p$ from the $L^p(w)$-boundedness of $T$. This is perhaps unexpected in view of the many known connections between the two questions.
\end{example}

\begin{example}[Stein's non-degeneracy violated all over the place]\label{ex:ndCZK}
In the two previous examples, a variant of Stein condition would still be satisfied, if we only demanded \eqref{eq:ndStein} for $t\in(0,\infty)$. This final (arguably somewhat artificial) example shows that we can make \eqref{eq:ndStein} fail for a significantly larger set of $t\in\R$, while still retaining non-degeneracy in the sense of Definition \ref{def:ndCZK}\eqref{it:ndVar}.

Let $d\geq 2$ and $\varphi$ be as in \eqref{eq:reso}.
Let $(w_k)_{k\in\Z}$ be a sequence of unit vectors that is dense in the unit sphere $S^{d-1}$ of $\R^d$, and let $(v_j)_{j\in\Z}$ be a sequence that, for each $w_k$, contains arbitrarily long subsequences of constant value $w_k$. Fixing a resolution as in \eqref{eq:reso}, let finally
\begin{equation*}
  K(x):=\sum_{j\in\Z}\varphi(2^{-j}\abs{x})\frac{x\cdot v_j}{\abs{x}^{d+1}}.
\end{equation*}
It is immediate that $K$ satisfies not only the $\omega$-Calder\'on--Zygmund estimates with $\omega(t)=t$, but in fact the higher Calder\'on--Zygmund estimates $\abs{\partial^\alpha K(x)}\leq c_\alpha\abs{x}^{-d-\abs{\alpha}}$ of any order, and also that $K$ has vanishing integral over any sphere centred at the origin. It is well-known (see again \cite[Proposition II.5.5]{GCRF}) that, under these conditions, $K$ is the convolution kernel of a singular integral operator bounded on $L^2(\R^d)$.

To see that $K$ satisfies Definition \ref{def:ndCZK}\eqref{it:ndVar}, given $r>0$, let $r\leq 2^k<2r$, and $x:=2^k v_k\in B(0,r)^c$. Then
\begin{equation*}
  K(x)=\frac{2^k v_k\cdot v_k}{\abs{2^k v_k}^{d+1}}=2^{-kd}>(2r)^{-d}.
\end{equation*}
On the other hand, let us fix some candidate unit-vector $u_0$ and $a>0$ for Stein's condition \eqref{eq:ndStein}, and choose another unit vector $u_1\perp u_0$. By density, we can find some $w_k$ with $\abs{w_k-u_1}<\frac12 a$. Given $N>0$ (large), we can find $v_n,v_{n+1},\ldots,v_{n+N}\equiv w_k$. Then
\begin{equation*}
  K(x)\equiv\frac{x\cdot w_k}{\abs{x}^{d+1}}\quad\text{whenever }2^{n}\leq\abs{x}\leq 2^{n+N},
\end{equation*}
and in particular
\begin{equation*}
  \abs{K(t\cdot u_0)}=\frac{\abs{t u_0\cdot w_k}}{\abs{t u_0}^{d+1}}=\abs{t}^{-d}\abs{u_0\cdot(w_k-u_1)}\leq\frac12 a \abs{t}^{-d}\quad\text{when }
  \abs{t}\in[2^n,2^{n+N}].
\end{equation*}
So not only is \eqref{eq:ndStein} violated, but it is violated on symmetric line-segments that may be arbitrarily long relative to their distance from $0$. On the other hand, the points of non-degeneracy for Definition \ref{def:ndCZK}\eqref{it:ndVar}, $x=2^k v_k$ where $v_k$ are dense in $S^{d-1}$, have a rather wild distribution in the underlying space $\R^d$.
\end{example}

\subsection{Consequences of non-degeneracy}

We will use the assumption of non-degeneracy through the following result:

\begin{proposition}\label{prop:unifNonDegImplies}
Let $K$ be a non-degenerate Calder\'on--Zygmund kernel.
Then for every $A\geq 3$ and every ball $B=B(y_0,r)$, there is a disjoint ball $\tilde B=B(x_0,r)$ at distance $\dist(B,\tilde B)\eqsim Ar$ such that
\begin{equation}\label{eq:ndKcenter} 
    \abs{K(x_0,y_0)}\eqsim \frac{1}{A^d r^d},
\end{equation}
and  for all $y_1\in B$ and $x_1\in\tilde B$, we have
\begin{equation}\label{eq:ndKdiff}
  \int_B\abs{K(x_1,y)-K(x_0,y_0)}\ud y+\int_{\tilde B}\abs{K(x,y_1)-K(x_0,y_0)}\ud x\lesssim \frac{\eps_A}{A^d},
\end{equation}
where $\eps_A\to 0$ as $A\to\infty$.

The implied constants can depend at most on $c_K,\omega$ and $d$, as well as $c_0$ or $\abs{\Omega(\theta_0)}$ from Definition \ref{def:ndCZK}.
If $K$ is homogeneous, we can take $x_0=y_0+Ar\theta_0$.
\end{proposition}

\begin{proof}[Proof of Proposition \ref{prop:unifNonDegImplies}, case \eqref{it:ndVar}]
We assume that $K$ is as in Definition \ref{def:ndCZK}\eqref{it:ndVar}.
Fix a ball $B=B(y_0,r)$ and $A\geq 3$. We apply the assumption with $y_0$ in place of $y$ and $Ar$ in place of $r$. This produces a point $x_0\in B(y_0,Ar)^c$ such that
\begin{equation*}
  \frac{1}{c_0 (Ar)^d}\leq\abs{K(x_0,y_0)}\leq\frac{c_K}{\abs{x_0-y_0}^d}.
\end{equation*}
Let $\tilde B:=B(x_0,r)$. Then
\begin{equation*} 
  Ar\leq\abs{x_0-y_0}\leq (c_0 c_K)^{1/d}Ar,\quad\dist(B,\tilde B)\eqsim\abs{x_0-y_0}.
\end{equation*}
Moreover, if $x\in \tilde B$ and $y\in B$, then
\begin{equation*}
\begin{split}
  \abs{K(x,y)-K(x_0,y_0)}
  &\leq \abs{K(x,y)-K(x,y_0)}+\abs{K(x,y_0)-K(x_0,y_0)} \\
  &\leq \frac{1}{\abs{x-y_0}^d}\omega\Big(\frac{\abs{y-y_0}}{\abs{x-y_0}}\Big)
  +\frac{1}{\abs{x_0-y_0}^d}\omega\Big(\frac{\abs{x-y_0}}{\abs{x_0-y_0}}\Big) \\
  &\leq\frac{1}{(Ar-r)^d}\omega\Big(\frac{r}{Ar-r}\Big)+\frac{1}{(Ar)^d}\omega\Big(\frac{r}{Ar}\Big) \\
  &=\frac{1}{(Ar)^d}\Big[\frac{1}{(1-A^{-1})^d}\omega\Big(\frac{1}{A-1}\Big)+\omega\Big(\frac{1}{A}\Big)\Big] 
  =\frac{\eps_A}{(Ar)^d},
\end{split}
\end{equation*}
where $\eps_A\to 0$ as $A\to\infty$ by the condition that $\omega(t)\to 0$ as $t\to 0$. Integrating this over $x\in\tilde B$ or $y\in B$, which both have measure $\abs{\tilde B}=\abs{B}\eqsim r^d$, we obtain \eqref{eq:ndKdiff}.
\end{proof}

\begin{proof}[Proof of Proposition \ref{prop:unifNonDegImplies}, case \eqref{it:ndRough}]
We assume that $K$ is as in Definition \ref{def:ndCZK}\eqref{it:ndRough}. Fix a ball $B=B(y_0,r)$ and $A\geq 3$. Let $x_0=y_0+Ar\theta_0$ and $\tilde B=B(x_0,r)$. Clearly $\dist(\tilde B,B)=(A-2)r\eqsim Ar$ and 
\begin{equation*}
  \abs{K(x_0,y_0)}=\frac{\abs{\Omega(x_0-y_0)}}{\abs{x_0-y_0}^d}
    =\frac{\abs{\Omega(Ar\theta_0)}}{\abs{Ar\theta_0}^d}
    =\frac{\abs{\Omega(\theta_0)}}{(Ar)^d}
    \eqsim\frac{1}{(Ar)^d},
\end{equation*}
recalling that the implied constant was allowed to depend on $\abs{\Omega(\theta_0)}$.

We then consider the integrals in \eqref{eq:ndKdiff}. Writing $x\in B(x_0,r)=B(y_0+Ar\theta_0,r)$ as $x=y_0+Ar\theta_0+ru$ and $y\in B(y_0,r)$ as $y=y_0+rv$, where $u,v\in B(0,1)$, and using the homogeneity of $\Omega$, we have
\begin{equation*}
\begin{split}
  K(x,y)-K(x_0,y_0)
  &=\frac{\Omega(x-y)}{\abs{x-y}^d}-\frac{\Omega(x_0-y_0)}{\abs{x_0-y_0}^d} \\
  &=\frac{\Omega(Ar\theta_0+r(u-v))}{\abs{Ar\theta_0+r(u-v)}^d}-\frac{\Omega(Ar\theta_0)}{\abs{Ar\theta_0}^d} \\
  &=\frac{1}{(Ar)^d}\Big(\frac{\Omega(\theta_0+A^{-1}(u-v))}{\abs{\theta_0+A^{-1}(u-v)}^d}-\Omega(\theta_0)\Big),
\end{split}
\end{equation*}
where
\begin{equation*}
\begin{split}
  &\frac{\Omega(\theta_0+A^{-1}(u-v))}{\abs{\theta_0+A^{-1}(u-v)}^d}-\Omega(\theta_0) \\
  &=\frac{\Omega(\theta_0+A^{-1}(u-v))-\Omega(\theta_0)}{\abs{\theta_0+A^{-1}(u-v)}^d}+\Omega(\theta_0)\Big(\frac{1}{\abs{\theta_0+A^{-1}(u-v)}^d}-1\Big)=:I+II.
\end{split}
\end{equation*}
Here it is immediate that $\abs{II}\lesssim A^{-1}$, and hence the integral of $(Ar)^{-d}II$ over either $x\in\tilde B$ or $y\in B$ is bounded by $A^{-d-1}=A^{-d}\eps_A$.

We turn to term $I$. Keeping either $x\in\tilde B$ fixed and varying $y\in B$, or the other way round, the difference $u-v$ varies over a subset of $B(0,2)$. Hence both $\int_{\tilde B}(Ar)^{-d}\abs{I}\ud x$ and $\int_B (Ar)^{-d}\abs{I}\ud y$ are dominated by
\begin{equation*}
\begin{split}
  &   \frac{1}{(Ar)^d}\int_{B(0,2)}\Babs{\frac{\Omega(\theta_0+A^{-1}z)-\Omega(\theta_0)}{\abs{\theta_0+A^{-1}z}^d}}r^d\ud z \\
  &\lesssim A^{-d}\fint_{B(0,2/A)}\abs{\Omega(\theta_0+s)-\Omega(\theta_0)}\ud s = A^{-d}\eps_A
\end{split}
\end{equation*}
by the assumption that $\theta_0$ is a Lebesgue point of $\Omega$.
\end{proof}

\subsection{Approximate weak factorisation}

For the class of non-degenerate Calder\'on--Zygmund operators just described, we prove certain ``weak factorisation'' type results that are pivotal in our proof of Theorem \ref{thm:summary}. These results have a technical flavour and may fail to have an ``independent interest'', but they are precisely what we need below. For a ball $B\subset\R^d$, we denote
\begin{equation*}
\begin{split}
    L^\infty(B) &=\{f\in L^\infty(\R^d): f=1_B f\},\\ L^\infty_0(B) &=\{f\in L^\infty(B):\int_B f=0\},\\
    L^\infty_+(B) &=\{f\in L^\infty(B): f\geq 0\}.
\end{split}
\end{equation*}

\begin{lemma}\label{lem:dec1}
Let $T$ be a non-degenerate Calder\'on--Zygmund operator. Using the notation of Proposition \ref{prop:unifNonDegImplies},
if $f\in L^\infty_0(B)$ and $g\in L^\infty_+(\tilde B)$ is such that $\Norm{g}{\infty}\lesssim\fint_{\tilde B}g$, then there is a decomposition
\begin{equation*}
  f=g Th-h T^* g+\tilde{f},
\end{equation*}
where $\tilde{f}\in L^\infty_0(\supp g)$ and $h\in L^\infty(\supp f)$ satisfy
\begin{equation*}
  \Norm{g}{\infty}\Norm{h}{\infty}\lesssim A^d\Norm{f}{\infty},\qquad
  \Norm{\tilde f}{\infty}\lesssim\eps_A\Norm{f}{\infty},
\end{equation*}
provided that $A$ is chosen large enough so that $\eps_A\ll 1$.
\end{lemma}

\begin{proof}
The decomposition is given by
\begin{equation*}
  f=\frac{f}{T^* g}T^*g
   =:-h T^*g
   =-h T^*g+gTh-gTh
   =:-h T^*g+gTh+\tilde{f},
\end{equation*}
where we need to justify that the definition of $h:=-f/T^*g$ does not involve division by zero. However, if $y\in B$, then
\begin{equation*}
\begin{split}
  T^*g(y)
  &=\int_{\tilde B}K(x,y)g(x)\ud x \\
  &=K(x_0,y_0)\int_{\tilde B}g(x)\ud x+\int_{\tilde B}[K(x,y)-K(x_0,y_0)]g(x)\ud x=I+II,
\end{split}
\end{equation*}
where, using Proposition \ref{prop:unifNonDegImplies},
\begin{equation*}
  \abs{I}\eqsim \frac{1}{A^d r^d}\int_{\tilde B}g\eqsim\frac{1}{A^d}\fint_{\tilde B}g\eqsim\frac{1}{A^d}\Norm{g}{\infty}
\end{equation*}
and
\begin{equation*}
   \abs{II}\lesssim \int_{\tilde B}\abs{K(x,y)-K(x_0,y_0)}\ud x \Norm{g}{\infty}\lesssim \frac{\eps_A}{A^d}\Norm{g}{\infty},
 \end{equation*}
so that
\begin{equation*}
  \abs{T^*g(y)}=\abs{I+II}\geq\abs{I}-\abs{II}\gtrsim \frac{1}{A^d}\Norm{g}{\infty},
\end{equation*}
recalling that $A$ was chosen large enough so that $\eps_A\ll 1$. This justifies the well-definedness of the decomposition, and we turn to the quantitative bounds.

From the previous considerations it directly follows that
\begin{equation*}
  \Norm{g}{\infty}\Norm{h}{\infty}\lesssim\Norm{g}{\infty}\frac{\Norm{f}{\infty}}{A^{-d}\Norm{g}{\infty}}=A^d \Norm{f}{\infty}.
\end{equation*}
It is also immediate that
\begin{equation*}
  -\int_{\tilde B}\tilde{f}
  =\int gTh=\int hT^*g=\int \frac{f}{T^*g}T^*g=\int_B f=0.
\end{equation*}
Let us then estimate
\begin{equation*}
  Th=T\Big(\frac{f}{T^*g}\Big)
  =T\Big(\frac{f}{T^*g}-\frac{f}{K(x_0,y_0)\int_{\tilde B}g}\Big)+\frac{1}{K(x_0,y_0)\int_{\tilde B}g}Tf=:I'+II'.
\end{equation*}
For $y\in B$,
\begin{equation*}
\begin{split}
  &\Babs{\frac{1}{T^*g(y)}-\frac{1}{K(x_0,y_0)\int_{\tilde B}g(y)}}
  =\Babs{\frac{K(x_0,y_0)\int_{\tilde B}g-T^* g(y)}{T^*g(y)K(x_0,y_0)\int_{\tilde B}g}} \\
  &\lesssim\frac{1}{(A^{-d}\Norm{g}{\infty})^2}\int_{\tilde B}\abs{K(x_0,y_0)-K(x,y)}\abs{g(x)}\ud x \\
  &\lesssim\frac{1}{(A^{-d}\Norm{g}{\infty})^2}\frac{\eps_A}{A^d}\Norm{g}{\infty}=\frac{A^d \eps_A}{\Norm{g}{\infty}}.
\end{split}
\end{equation*}
Hence for $x\in\tilde B$,
\begin{equation*}
\begin{split}
  \abs{I'(x)}
  &\leq\int_{B}\abs{K(x,y)}\Norm{f}{\infty}\frac{A^d \eps_A}{\Norm{g}{\infty}}\ud y \\
  &\lesssim \fint_{B}\frac{1}{A^d}\Norm{f}{\infty}\frac{A^d \eps_A}{\Norm{g}{\infty}}\ud y 
  =\eps_A\frac{\Norm{f}{\infty}}{\Norm{g}{\infty}}.
\end{split}
\end{equation*}
On the other hand, recalling that $f\in L^\infty_0(B)$,
\begin{equation*}
\begin{split}
  \abs{Tf(x)}
  &=\Babs{\int_B [K(x,y)-K(x_0,y_0)]f(y)\ud y} \\
  &\leq\int_B\abs{K(x,y)-K(x_0,y_0)}\ud y\Norm{f}{\infty}
  \lesssim \frac{\eps_A}{A^d}\Norm{f}{\infty},
\end{split}
\end{equation*}
and thus
\begin{equation*}
  \abs{II'(x)}\lesssim\frac{\abs{Tf(x)}}{A^{-d}\Norm{g}{\infty}}\lesssim\eps_A\frac{\Norm{f}{\infty}}{\Norm{g}{\infty}}.
\end{equation*}
It is then immediate that
\begin{equation*}
  \Norm{\tilde f}{\infty}
  =\Norm{gTh}{\infty}
  \lesssim\Norm{g}{\infty}\eps_A\frac{\Norm{f}{\infty}}{\Norm{g}{\infty}}=\eps_A\Norm{f}{\infty}.\qedhere
\end{equation*}
\end{proof}

By iterating the previous decomposition (but just once more), we achieve the useful additional property that the error term is supported on the same set as the original function. In the following lemma and below, we will make use of the following notion:

\begin{definition}[Major subset]\label{def:major}
If $E\subset F\subset\R^d$ are sets of finite measure, we say that $E$ is a {\em major subset} of $F$ if $\abs{E}\geq c\abs{F}$ for some fixed constant $c\in(0,1)$ that  depends only on the admissible parameters, as described in Section \ref{ss:notation}.
\end{definition}

In the following lemma, we denote certain major subsets by the suggestive letter $Q$, since the main subsequent application deals with the case, where these sets a cubes; however, the lemma itself does not require assuming this.

\begin{lemma}\label{lem:dec2}
Let $T$ be a non-degenerate Calder\'on--Zygmund operator.
Let $B$ and $\tilde B$ be as in Proposition \ref{prop:unifNonDegImplies}, and $Q\subset B$, $\tilde Q\subset\tilde B$ be their major subsets, i.e., $\abs{Q}\gtrsim\abs{B}$ and $\abs{\tilde Q}\gtrsim\abs{\tilde B}$.

If $f\in L^\infty_0(Q)$, there is a decomposition
\begin{equation}\label{eq:dec2}
  f=\sum_{i=1}^2 (g_i T h_i-h_i T^* g_i)+\tilde{\tilde f},
\end{equation}
where $\tilde{\tilde f}\in L^\infty_0(Q)$, $g_i\in L^\infty(\tilde Q)$ and $h_i\in L^\infty(Q)$ satisfy
\begin{equation*}
  \Norm{g_i}{\infty}\Norm{h_i}{\infty}\lesssim A^d\Norm{f}{\infty},\qquad
  \Norm{\tilde{\tilde f}}{\infty}\lesssim\eps_A\Norm{f}{\infty},
\end{equation*}
provided that $A$ is chosen large enough so that $\eps_A\ll 1$.
\end{lemma}

\begin{proof}
We first apply Lemma \ref{lem:dec1} to $f$ and $g_1:=1_{\tilde Q}\in L^\infty_+(\tilde B)$, which clearly satisfies the condition $\Norm{g}{\infty}=1\lesssim\abs{\tilde Q}/\abs{\tilde B}=\fint_{\tilde B}g$. Thus Lemma \ref{lem:dec1} yields a decomposition
\begin{equation*}
  f=g_1Th_1-h_1T^*g_1+\tilde{f},
\end{equation*}
where $\tilde{f}\in L^\infty_0(\supp g_1)=L^\infty_0(\tilde Q)$, $g_1\in L^\infty(\tilde Q)$ and $h_1\in L^\infty(\supp f)\subset L^\infty(Q)$ with the estimates
\begin{equation*}
 \Norm{g_1}{\infty}\Norm{h_1}{\infty}\lesssim A^d\Norm{f}{\infty},\qquad\Norm{\tilde f}{\infty}\lesssim \eps_A\Norm{f}{\infty}.
\end{equation*}

We then wish to apply Lemma \ref{lem:dec1} again, this time to the functions $\tilde f$ and $\tilde g:=1_Q\in L^\infty_+(B)$, and the adjoint operator $T^*$ in place of $T$. For this, we notice that the conclusions of Proposition \ref{prop:unifNonDegImplies} are preserved under the replacement of $(B,\tilde B,T)$ by $(\tilde B,B,T^*)$. Hence Lemma \ref{lem:dec1} provides a decomposition
\begin{equation*}
  \tilde{f}=\tilde{g}T^*\tilde{h}-\tilde{h}T\tilde{g}+\tilde{\tilde f},
\end{equation*}
where $\tilde{\tilde f}\in L^\infty_0(\supp\tilde g)=L^\infty_0(Q)$, $\tilde g\in L^\infty(Q)$ and $\tilde h\in L^\infty(\supp\tilde f)\subset L^\infty(\tilde Q)$ with the estimates
\begin{equation*}
  \Norm{\tilde g}{\infty}\Norm{\tilde h}{\infty}\lesssim A^d\Norm{\tilde f}{\infty}\lesssim A^d\Norm{f}{\infty},\qquad
  \Norm{\tilde{\tilde f}}{\infty}\lesssim\eps_A\Norm{\tilde f}{\infty}\lesssim\eps_A\Norm{f}{\infty}.
\end{equation*}
(We could write $\eps_A^2$ in the ultimate right, but since $\eps_A\to 0$ at an unspecified rate anyway, this is irrelevant.)
It remains to define $g_2:=-\tilde h\in L^\infty(\tilde Q)$, $h_2:=\tilde g\in L^\infty(Q)$ so that
\begin{equation*}
  \tilde{g}T^*\tilde{h}-\tilde{h}T\tilde{g}
  =-h_2 T^* g_2+g_2 T h_2,
\end{equation*}
and we get the required decomposition \eqref{eq:dec2}.
\end{proof}

\subsection{Necessary conditions for $[b,T]:L^p\to L^q$ when $p\leq q$}
We now come to the proof of some of the ``only if'' direction of Theorem \ref{thm:summary}. Assuming a weak form of the boundedness of the commutator $[b,T]$, we wish to derive the membership of $b$ in a suitable function space, with estimates for its norm. The relevant spaces here will be the functions of bounded mean oscillation,
\begin{equation*}
  \BMO(\R^d):=\Big\{b\in L^1_{\loc}(\R^d)\Big|\Norm{b}{\BMO}:=\sup_B\fint_B\abs{b-\ave{b}_B}<\infty\Big\},
\end{equation*}
where the supremum is over all balls $B\subset\R^d$, and the homogeneous H\"older spaces
\begin{equation*}
  \dot C^{0,\alpha}(\R^d):=\Big\{b:\R^d\to\C\Big|\Norm{b}{\dot C^{0,\alpha}}:=\sup_{x\neq y}\frac{\abs{b(x)-b(y)}}{\abs{x-y}^\alpha}<\infty\Big\}.
\end{equation*}
Note that we do not impose any boundedness condition on $b$; this would lead to the inhomogeneous H\"older space $C^{0,\alpha}$, which does not play any role in our results.

\begin{theorem}\label{thm:pleq1}
Let $K$ be a non-degenerate Calder\'on--Zygmund kernel, and $b\in L^1_{\loc}(\R^d)$.
Let further
\begin{equation*}
   1<p\leq q<\infty,\qquad \alpha:=d\Big(\frac{1}{p}-\frac{1}{q}\Big)\geq 0,
\end{equation*}
and suppose that $[b,T]$ satisfies the following weak form of $L^p\to L^q$ boundedness:
\begin{equation}\label{eq:veryWeakLpLq}
\begin{split}
  \abs{\pair{[b,T]f}{g}}
  &=\Babs{\iint (b(x)-b(y))K(x,y)f(y)g(x)\ud y\ud x} \\
  &\leq \Theta\cdot \Norm{f}{\infty}\abs{B}^{1/p}\cdot\Norm{g}{\infty}\abs{\tilde B}^{1/q'},
\end{split}
\end{equation}
whenever $f\in L^\infty(B)$, $g\in L^\infty(\tilde B)$ for any two balls of equal radius $r$ and distance $\dist(B,\tilde B)\gtrsim r$.
Then
\begin{itemize}
  \item if $\alpha=0$, equivalently $p=q$, we have  $b\in\BMO(\R^d)$, and
  $\Norm{b}{\BMO}\lesssim \Theta;$ 
  \item if $\alpha\in(0,1]$, we have $b\in\dot C^{0,\alpha}(\R^d)$, and
  $\Norm{b}{\dot C^{0,\alpha}}\lesssim \Theta;$ 
  \item if $\alpha>1$, the function $b$ is constant, so in fact $[b,T]=0$.
\end{itemize}
\end{theorem}

\begin{proof}
Let us consider a fixed ball $B\subset\R^d$ of radius $r$. Then
\begin{equation*}
  \fint_B\abs{b-\ave{b}_B}\eqsim\sup_{\substack{f\in L^\infty_0(B) \\ \Norm{f}{\infty}\leq 1}}\Babs{\fint_B bf}
\end{equation*}
is finite by the assumption that $b\in L^1_{\loc}(\R^d)$. Given $f\in L^\infty_0(B)$, we apply Lemma \ref{lem:dec2} to write
\begin{equation*}
  f=\sum_{i=1}^2(g_i Th_i-h_i T^* g_i)+\tilde{\tilde f},
\end{equation*}
where $\tilde{\tilde f}\in L^\infty_0(B)$, $g_i\in L^\infty(\tilde B)$ and $h_i\in L^\infty(B)$ satisfy
\begin{equation*}
   \Norm{g_i}{\infty}\Norm{h_i}{\infty}\lesssim A^d\Norm{f}{\infty},\qquad\Norm{\tilde{\tilde f}}{\infty}\lesssim \eps_A\Norm{f}{\infty},
\end{equation*}
and $\tilde B$ is another ball of radius $r$ such that $\dist(B,\tilde B)\eqsim Ar$.

Then
\begin{equation*}
\begin{split}
  \int bf
  &=\sum_{i=1}^2 \int b(g_iTh_i-h_i T^* g_i)+\int b\tilde{\tilde f} \\
  &=\sum_{i=1}^2 \int [g_i bTh_i-g_i T(b h_i)]+\int b\tilde{\tilde f}
  =\sum_{i=1}^2 \int g_i[b,T]h_i+\int b\tilde{\tilde f},
\end{split}
\end{equation*}
where, by assumption \eqref{eq:veryWeakLpLq},
\begin{equation*}
  \Babs{\int g_i[b,T]h_i}
  \leq \Theta\cdot \Norm{g_i}{\infty}\Norm{h_i}{\infty}\cdot\abs{B}^{1/p+1/q'}
  \lesssim \Theta\cdot A^d\Norm{f}{\infty}\cdot \abs{B}\cdot r^{\alpha}.
\end{equation*}
Thus
\begin{equation*}
  \Babs{\fint_B bf}
  \lesssim \Theta\cdot A^d\Norm{f}{\infty}r^\alpha+\Babs{\fint_B b\tilde{\tilde f}},
\end{equation*}
where
\begin{equation*}
  \Babs{\fint_B b\tilde{\tilde f}}
  \leq\Big(\fint_B\abs{b-\ave{b}_B}\Big)\Norm{\tilde{\tilde f}}{\infty}
  \lesssim\Big(\fint_B\abs{b-\ave{b}_B}\Big)\eps_A\Norm{f}{\infty}.
\end{equation*}
Taking the supremum over $f\in L^\infty_0(B)$ of norm one, we deduce that
\begin{equation*}
  \fint_B\abs{b-\ave{b}_B}\lesssim \Theta A^d r^\alpha+\eps_A\fint_B\abs{b-\ave{b}_B},
\end{equation*}
and the last term can be absorbed if $A$ is fixed large enough, depending only on the implied constants. Thus
\begin{equation*}
  \fint_B\abs{b-\ave{b}_B}\lesssim \Theta r^\alpha.
\end{equation*}

If $\alpha=0$, this is precisely the condition $b\in\BMO(\R^d)$ with the claimed estimate.

For $\alpha>0$, this is also a well-known reformulation of $b\in\dot C^{0,\alpha}(\R^d)$ (which consists only of constants for $\alpha>1$). We recall the argument for completeness.

Let $x_i$, $i=1,2$, be two Lebesgue points of $b$, and let $r:=\abs{x_1-x_2}$. Then
\begin{equation*}
  b(x_i)=\lim_{t\to 0}\ave{b}_{B(x_i,t)}
  =\sum_{k=0}^\infty\big(\ave{b}_{B(x_i, 2^{-k-1}r)}-\ave{b}_{B(x_i,2^{-k}r)}\big)+\ave{b}_{B(x_i,r)}.
\end{equation*}
If $B\subset B^*$ are two balls of radius comparable to $R$, then
\begin{equation}\label{eq:2balls}
  \abs{\ave{b}_B-\ave{b}_{B^*}}
  =\Babs{\fint_B(b-\ave{b}_{B^*})}
  \leq\fint_B\abs{b-\ave{b}_{B^*}}
  \lesssim\fint_{B^*}\abs{b-\ave{b}_{B^*}} \lesssim \Theta R^\alpha,
\end{equation}
and thus
\begin{equation*}
  \sum_{k=0}^\infty\abs{\ave{b}_{B(x_i, 2^{-k-1}r)}-\ave{b}_{B(x_i,2^{-k}r)}}
  \lesssim\sum_{k=0}^\infty \Theta (2^{-k}r)^\alpha
  \lesssim \Theta r^\alpha.
\end{equation*}
Hence
\begin{equation*}
  \abs{b(x_1)-b(x_2)}
  \lesssim \Theta r^\alpha+\abs{\ave{b}_{B(x_1,r)}-\ave{b}_{B(x_2,r)}},
\end{equation*}
where another application of \eqref{eq:2balls} shows that
\begin{equation*}
   \abs{\ave{b}_{B(x_i,r)}-\ave{b}_{B(\frac12(x_1+x_2),2r)}}
   \lesssim \Theta r^\alpha.
\end{equation*}
Thus altogether
\begin{equation*}
   \abs{b(x_1)-b(x_2)}
  \lesssim \Theta r^\alpha
  =\Theta\abs{x_1-x_2}^\alpha,
\end{equation*}
and this can be extended to all $x_1,x_2$ by redefining $b$ in a set of measure zero.
This is the required bound for $\Norm{b}{\dot C^{0,\alpha}}$ if $\alpha\in(0,1]$.

If $\alpha>1$, we let $y_k:=x_1+N^{-1}k(x_2-x_1)$ for $k=0,1,\ldots,N$ to deduce that
\begin{equation*}
\begin{split}
  \abs{b(x_1)-b(x_2)}
  \leq\sum_{k=1}^N\abs{b(y_k)-b(y_{k-1})} 
  \lesssim\sum_{k=1}^N \Theta (N^{-1}r)^\alpha 
  =N^{1-\alpha}\Theta r^\alpha.
\end{split}
\end{equation*}
With $N\to\infty$, this shows that $b(x_1)=b(x_2)$, and hence $b$ is constant.
\end{proof}

\subsection{Necessary condition for $[b,T]:L^p\to L^q$ when $p>q$}
We now come to the more exotic case of Theorem \ref{thm:summary}, which is precisely restated in the following:

\begin{theorem}\label{thm:offDiag}
Let $K$ be a non-degenerate Calder\'on--Zygmund kernel, and $b\in L^1_{\loc}(\R^d)$. Let
\begin{equation*}
  1<q<p<\infty,\qquad r=\frac{pq}{p-q}\in(1,\infty),
\end{equation*}
and suppose that $[b,T]$ satisfies the following weak form of $L^p\to L^q$ boundedness:
\begin{equation}\label{eq:quiteWeakLpLq}
  \sum_{i=1}^N\abs{\pair{[b,T]f_i}{g_i}}
  \leq \Theta\BNorm{\sum_{i=1}^N\Norm{f_i}{\infty} 1_{Q_i}}{p}
    \BNorm{\sum_{i=1}^N\Norm{g_i}{\infty} 1_{\tilde Q_i}}{q'},
\end{equation}
whenever, for each $i=1,\ldots,N$, we have $f_i\in L^\infty(Q_i)$ and $g_i\in L^\infty(\tilde Q_i)$ for cubes $Q_i$ and $\tilde Q_i$ such that $\dist(Q_i,\tilde Q_i)\gtrsim\diam(Q_i)=\diam(\tilde Q_i)$.

Then $b=a+c$ for some $a\in L^r(\R^d)$ and some constant $c\in\C$, where
$
  \Norm{a}{r}\lesssim\Theta.
$
\end{theorem}

Note that each term on the left of \eqref{eq:quiteWeakLpLq} can be defined as in \eqref{eq:veryWeakLpLq}.
In order to better understand the assumption \eqref{eq:quiteWeakLpLq}, we include:

\begin{lemma}
\begin{enumerate}
  \item For any $p,q\in(1,\infty)$, \eqref{eq:quiteWeakLpLq} follows if $[b,T]$ exists as a bounded linear operator $[b,T]: L^p\to L^q$, and $\Theta\eqref{eq:quiteWeakLpLq}\leq\Norm{[b,T]}{L^p\to L^q}$.
  \item If $p\leq q$, then \eqref{eq:quiteWeakLpLq} follows from \eqref{eq:veryWeakLpLq}, and $\Theta\eqref{eq:quiteWeakLpLq}\leq \Theta\eqref{eq:veryWeakLpLq}$.
\end{enumerate}
\end{lemma}

\begin{proof}
For certain fixed signs $\sigma_i$, and random signs $\eps_i$ on some probability space with expectation denoted by $\Exp$, we have
\begin{equation*}
\begin{split}
  \sum_{i=1}^N\abs{\pair{[b,T]f_i}{g_i}}
  &=\sum_{i=1}^N\sigma_i \pair{[b,T]f_i}{g_i}
  =\Exp\Bpair{[b,T]\sum_{i=1}^N\eps_i \sigma_i f_i}{\sum_{j=1}^N \eps_j g_j} \\
  &\leq\Exp\Norm{[b,T]}{L^p\to L^q}\BNorm{\sum_{i=1}^N\eps_i \sigma_i f_i}{p}\BNorm{\sum_{j=1}^N \eps_j g_j}{q'} \\
  &\leq \Norm{[b,T]}{L^p\to L^q}\BNorm{\sum_{i=1}^N\abs{f_i} }{p}\BNorm{\sum_{j=1}^N \abs{g_j} }{q'}.
\end{split}
\end{equation*}

If $p\leq q$, using \eqref{eq:veryWeakLpLq} followed by H\"older's inequality and several applications of $\Norm{\ }{\ell^s}\leq\Norm{\ }{\ell^t}$ if $t\leq s$, we find that
\begin{equation*}
\begin{split}
  \sum_{i=1}^N &\abs{\pair{[b,T]f_i}{g_i}}
  \leq\sum_{i=1}^N \Theta\Norm{f_i}{p}\Norm{g_i}{q'}
  \leq \Theta\Big(\sum_{i=1}^N\Norm{f_i}{p}^q\Big)^{1/q}\Big(\sum_{i=1}^N\Norm{g_i}{q'}^{q'}\Big)^{1/q'} \\
  &\leq \Theta\Big(\sum_{i=1}^N\Norm{f_i}{p}^p\Big)^{1/p}\Big(\sum_{i=1}^N\Norm{g_i}{q'}^{q'}\Big)^{1/q'} \\
  &= \Theta\BNorm{\Big(\sum_{i=1}^N\abs{f_i}^p\Big)^{1/p}}{p}\BNorm{\Big(\sum_{i=1}^N\abs{g_i}^{q'}\Big)^{1/q'}}{q'} 
  \leq \Theta\BNorm{\sum_{i=1}^N\abs{f_i}}{p}\BNorm{\sum_{i=1}^N\abs{g_i}}{q'},
\end{split}
\end{equation*}
where $\Theta=\Theta\eqref{eq:veryWeakLpLq}$.
\end{proof}

For the proof of Theorem \ref{thm:offDiag}, we need the following lemma. Given a cube $Q_0\subset\R^d$, we denote by $\mathscr D(Q_0)$ the collection of its dyadic subcubes (obtained by repeatedly bisecting each side of the initial cube any finite number of times).

\begin{lemma}\label{lem:dec3}
Let $f\in L^\infty_0(Q_0)$ for some cube $Q_0\subset\R^d$. Then it has a decomposition
\begin{equation*}
  f=\sum_{n=0}^N f_n,\qquad f_n=\sum_{k=0}^\infty f_{n,k},
\end{equation*}
where $f_{n,k}\in L^\infty_0(Q_{n,k})$ and $\Norm{f_{n,k}}{\infty}\lesssim\ave{\abs{f}}_{Q_{n,k}}$, and $Q_{n,k}\in\mathscr D(Q_0)$ are disjoint in $k$ for each $n$. Moreover, for all $n$ and $k$ we have $Q_{n,k}\subset Q_{n-1,j}$ for a unique $j$, and $\ave{\abs{f}}_{Q_{n,k}}>2\ave{\abs{f}}_{Q_{n-1,j}}$.
\end{lemma}

\begin{proof}
Let $\mathcal F_0=\{Q_0\}$ and
\begin{equation*}
  \mathcal F_{n+1}:=\bigcup_{F\in\mathcal F_n}\operatorname{ch}_{\mathcal F}F,\quad
  \operatorname{ch}_{\mathcal F}F
  :=\{Q\in\mathscr D(F)\text{ maximal with }\ave{\abs{f}}_Q>2\ave{\abs{f}}_F\}.
\end{equation*}
Since $f\in L^\infty(Q_0)$, there is a finite $N$ such that $\mathcal F_n=\varnothing$ for all $n>N$. For $F\in\mathcal F=\bigcup_{n=0}^N\mathcal F_n$, we define
\begin{equation*}
  E(F):=F\setminus\bigcup_{F'\in\operatorname{ch}_{\mathcal F}F}F'
\end{equation*}
and then
\begin{equation*}
  f_F:=1_{E(F)}f+\sum_{F'\in\operatorname{ch}_{\mathcal F}F}1_{F'}\ave{f}_{F'}-1_F\ave{f}_F,
\end{equation*}
so that $\int f_F=0$ and
\begin{equation*}
  \abs{f_F}\leq 1_{E(F)}2\ave{\abs{f}}_F+\sum_{F'\in\operatorname{ch}_{\mathcal F}F}1_{F'}2\cdot 2^d\ave{\abs{f}}_{F}+1_F\ave{f}_F
  \lesssim 1_F\ave{\abs{f}}_F.
\end{equation*}
Letting $(f_{n,k},Q_{n,k})_{k=0}^\infty$ be some enumeration of $(f_F,F)_{F\in\mathcal F_n}$, the claimed properties are easily checked.
\end{proof}

\begin{lemma}\label{lem:Doob}
Let $Q_k$ be cubes, and $E_k\subset Q_k$ their subsets with $\abs{E_k}\geq\eta\abs{Q_k}$ for some $\eta\in(0,1)$.Let $\lambda_k\geq 0$ and $p\in[1,\infty)$. Then
\begin{equation*}
  \BNorm{\sum_{k=0}^\infty\lambda_k 1_{Q_k}}{p}
  \lesssim \frac{1}{\eta}\BNorm{\sum_{k=0}^\infty\lambda_k 1_{E_k}}{p}
\end{equation*}
For $A\geq 1$, the bound is also true with $A^d$ in place of $1/\eta$, if $E_k=\tilde Q_k$ is another cube with $\dist(Q_k,\tilde Q_k)\leq A\ell(Q_k)=A\ell(\tilde Q_k)$.
\end{lemma}

In the second claim, a more delicate argument could be given to improve the bound $A^d$ to $\log A$, but this is unnecessary for the present purposes.

\begin{proof}
Dualising the left side with $\phi\in L^{p'}$, we find that
\begin{equation*}
  \int\Big(\sum_{k=0}^\infty\lambda_k 1_{Q_k}\Big)\phi
  =\sum_{k=0}^\infty\lambda_k \int_{Q_k}\phi
  \leq\sum_{k=0}^\infty\lambda_k \frac{\abs{E_k}}{\eta} \fint_{Q_k}\phi
  =\frac{1}{\eta}\int\Big(\sum_{k=0}^\infty\lambda_k 1_{E_k}\Big) M\phi,
\end{equation*}
and the first claim by the boundedness of the maximal operator on $L^{p'}$.

For the second claim, let $Q_k^*$ be a cube that contains both $Q_k$ and $\tilde Q_k$, with $\ell(Q_k^*)\lesssim A\ell(Q_k)$. Then we first use the trivial bound $1_{Q_k}\leq 1_{Q_k^*}$, and then the first part of the lemma with $Q_k^*$ in place of $Q_k$, and $\tilde Q_k\subset Q_k^*$ in place of $E_k$, observing that $\abs{\tilde Q_k}\gtrsim A^{-d}\abs{Q_k^*}$ .
\end{proof}

\begin{proof}[Proof of Theorem \ref{thm:offDiag}]
We fix a cube $Q_0\subset\R^d$ and consider the quantity
\begin{equation*}
  C_R:=\sup\Big\{\Babs{\int_{Q_0} bf}:f\in L^\infty_0(Q_0),\Norm{f}{\infty}\leq R,\Norm{f}{r'}\leq 1\Big\}.
\end{equation*}
This has the trivial a priori upper bound $C_R\leq\Norm{b}{L^1(Q_0)}R<\infty$, since $b\in L^1_{\loc}(R^d)$, but we wish to deduce a bound independent of $R$. To this end, we fix an $f\in L^\infty_0(Q_0)$ and make the decomposition given by Lemma \ref{lem:dec3}. Then
\begin{equation*}
  \int_{Q_0}bf=\sum_{n=0}^N\int_{Q_0}bf_n=\sum_{n=0}^N\sum_{k=0}^\infty\int_{Q_{n,k}}bf_{n,k};
\end{equation*}
the last step follows since $bf_n=\sum_{k=0}^\infty bf_{n,k}$ is integrable and the terms $bf_{n,k}$ are disjointly supported.
For each $(k,n)$, we apply the decomposition of Lemma \ref{lem:dec2} to write
\begin{equation*}
  f_{n,k}=\sum_{i=1}^2(g_{n,k}^i T h_{n,k}^i-h_{n,k}^i T^* g_{n,k}^i)+\tilde{\tilde f}_{n,k},
\end{equation*}
where $\tilde{\tilde f}_{n,k}\in L^\infty_0(Q_{n,k})$, $g_{n,k}^i\in L^\infty(\tilde Q_{n,k})$ and $h_{n,k}^i \in L^\infty(Q_{n,k})$ for some cubes $\tilde Q_{n,k}$ of the same size as $Q_{n,k}$ and distance $\dist(Q_{n,k},\tilde Q_{n,k})\eqsim A\diam(Q_{n,k})$. In particular, the functions $\tilde{\tilde f}_{n,k}\in L^\infty_0(Q_{n,k})$ are again disjointly supported with respect to $k$, for each fixed $n$. Thus
\begin{equation*}
  \int_{Q_{n,k}}bf_{n,k}
  =\sum_{i=1}^2\int g_{n,k}^i[b,T]h_{n,k}^i+\int_{Q_{n,k}}b\tilde{\tilde f}_{n,k}.
\end{equation*}
Since both the left side and the second term on the right is summable over $k$, so is the first term on the right, and we have
\begin{equation*}
 \int_{Q_0}bf_n=\sum_{k=0}^\infty\sum_{i=1}^2\int g_{n,k}^i[b,T]h_{n,k}^i+\int_{Q_0}b\tilde{\tilde f}_n,\qquad 
  \tilde{\tilde f}_n:=\sum_{k=0}^\infty \tilde{\tilde f}_{n,k}.
\end{equation*}
Summing over $n=0,1,\ldots,N$, we further deduce that
\begin{equation*}
  \int_{Q_0}bf=\sum_{n=0}^N \sum_{k=0}^\infty\sum_{i=1}^2\int g_{n,k}^i[b,T]h_{n,k}^i + \int_{Q_0}b \tilde{\tilde f},\qquad
   \tilde{\tilde f}:=\sum_{n=0}^N \tilde{\tilde f}_n.
\end{equation*}
We notice that
\begin{equation*}
\begin{split}
  \abs{\tilde{\tilde f}}
  &\leq\sum_{n=0}^N\sum_{k=0}^\infty \abs{\tilde{\tilde f}_{n,k}}
  \leq\sum_{n=0}^N\sum_{k=0}^\infty \Norm{\tilde{\tilde f}_{n,k}}{\infty}1_{Q_{n,k}}
  \lesssim\eps_A\sum_{n=0}^N\sum_{k=0}^\infty \Norm{f_{n,k}}{\infty}1_{Q_{n,k}} \\
&  \lesssim\eps_A\sum_{n=0}^N\sum_{k=0}^\infty \ave{\abs{f_{n,k}}}_{Q_{n,k}}1_{Q_{n,k}} 
  \lesssim\eps_A Mf.
\end{split}
\end{equation*}
This pointwise maximal function bound proves both
\begin{equation*}
  \Norm{\tilde{\tilde f}}{\infty}\lesssim\eps_A\Norm{f}{\infty}\leq \eps_A R,\qquad
  \Norm{\tilde{\tilde f}}{r'}\lesssim\eps_A\Norm{f}{r'}\leq \eps_A,
\end{equation*}
so that
\begin{equation*}
  \Babs{\int_{Q_0}b \tilde{\tilde f}}\lesssim \eps_A C_R.
\end{equation*}

On the other hand, by the definition of convergent series, we have
\begin{equation*}
  \sum_{n=0}^N \sum_{k=0}^\infty\sum_{i=1}^2\int g_{n,k}^i[b,T]h_{n,k}^i 
  =\lim_{K\to\infty}\sum_{n=0}^N \sum_{k=0}^K\sum_{i=1}^2\int g_{n,k}^i[b,T]h_{n,k}^i.
\end{equation*}
Recalling that $g_{n,k}^i\in L^\infty(\tilde Q_{n,k})$ and $h_{n,k}^i \in L^\infty(Q_{n,k})$, where $\dist(Q_{n,k},\tilde Q_{n,k})\eqsim A\diam(Q_{n,k})=A\diam(\tilde Q_{n,k})$, the finite triple sum has exactly the form appearing in \eqref{eq:quiteWeakLpLq}, and we can estimate
\begin{equation*}
  \Babs{\sum_{n=0}^N \sum_{k=0}^K\sum_{i=1}^2\int g_{n,k}^i[b,T]h_{n,k}^i}
  \leq \Theta\BNorm{\sum_{n,k,i}\Norm{g_{n,k}^i}{\infty}1_{\tilde Q_{n,k}}}{q'}\BNorm{\sum_{n,k,i}\Norm{h_{n,k}^i}{\infty}1_{Q_{n,k}}}{p}.
\end{equation*}
Note that $g_{n,k}^i$ and $h_{n,k}^i$ appear in the decomposition of $f_{n,k}^i$ in a bilinear way so that we are free to multiply these functions by any $\alpha>0$ and $\alpha^{-1}$, respectively. In particular, since $1/r=1/q-1/p$ implies that $1/r'=1/q'+1/p$, we may arrange the bound
\begin{equation*}
   \Norm{g_{n,k}^i}{\infty}\Norm{h_{n,k}^i}{\infty}\lesssim A^d\Norm{f_{n,k}}{\infty}
   \lesssim A^d\ave{\abs{f}}_{Q_{n,k}}
\end{equation*}
into the form
\begin{equation*}
  \Norm{g_{n,k}^i}{\infty}\lesssim A^d \ave{\abs{f}}_{Q_{n,k}}^{r'/q'},\qquad\Norm{h_{n,k}^i}{\infty}\lesssim \ave{\abs{f}}_{Q_{n,k}}^{r'/p}.
\end{equation*}
Thus
\begin{equation*}
  \sum_{n,k,i}\Norm{h_{n,k}^i}{\infty}1_{Q_{n,k}}
  \lesssim\sum_{n,k}\ave{\abs{f}}_{Q_{n,k}}^{r'/p}1_{Q_{n,k}}.
\end{equation*}
At a fixed point $x\in Q_{N,k_N}\subset\ldots\subset Q_{1,k_1}\subset Q_0$, the averages $\ave{\abs{f}}_{Q_{n,k_n}}$ satisfy $\ave{\abs{f}}_{Q_{n+1,k_{n+1}}}>2\ave{\abs{f}}_{Q_{n,k_n}}$; thus $\ave{\abs{f}}_{Q_{n,k_n}}\leq 2^{n-N}\ave{\abs{f}}_{Q_{N,k_N}}$, and hence
\begin{equation*}
  \sum_{n=0}^N \ave{\abs{f}}_{Q_{n,k_n}}^{r'/p}\leq\sum_{n=0}^N 2^{(n-N)r'/p}\ave{\abs{f}}_{Q_{N,k_N}}^{r'/p}\lesssim\ave{\abs{f}}_{Q_{N,k_N}}^{r'/p}\leq (Mf(x))^{r'/p},
\end{equation*}
so that
\begin{equation*}
  \sum_{n,k}\ave{\abs{f}}_{Q_{n,k}}^{r'/p}1_{Q_{n,k}}\lesssim (Mf)^{r'/p},
\end{equation*}
and hence
\begin{equation*}
  \BNorm{\sum_{n,k,i}\Norm{h_{n,k}^i}{\infty}1_{Q_{n,k}}}{p}
  \lesssim\Norm{(Mf)^{r'/p}}{p}
  =\Norm{Mf}{r'}^{r'/p}\lesssim\Norm{f}{r'}^{r'/p}\leq 1.
\end{equation*}
For the similar term involving the $g_{n,k}^i$, we need in addition Lemma \ref{lem:Doob}:
\begin{equation*}
\begin{split}
  \BNorm{\sum_{n,k,i}\Norm{g_{n,k}^i}{\infty}1_{\tilde Q_{n,k}}}{q'}
  \lesssim A^d\BNorm{\sum_{n,k,i}\Norm{g_{n,k}^i}{\infty}1_{Q_{n,k}}}{q'}  
  \lesssim A^{2d}\BNorm{\sum_{n,k,i}\ave{f_{n,k}}_{Q_{n,k}}^{r'/q'}1_{Q_{n,k}}}{q'},
\end{split}
\end{equation*}
where, as before,
\begin{equation*}
  \BNorm{\sum_{n,k,i}\ave{f_{n,k}}_{Q_{n,k}}^{r'/q'}1_{Q_{n,k}}}{q'}
  \lesssim\Norm{(Mf)^{r'/q'}}{q'} 
  =\Norm{Mf}{r'}^{r'/q'}
  \lesssim\Norm{f}{r'}^{r'/q'}\leq 1.
\end{equation*}

Collecting the bounds, we have proved that
\begin{equation*}
\begin{split}
  \Babs{\int_{Q_0}bf}
  &\leq\lim_{K\to\infty}\Babs{\sum_{n=0}^N\sum_{k=0}^K\sum_{i=1}^2\int g_{n,k}^i[b,T]h_{n,k}^i}+\Babs{\int_{Q_0}b\tilde{\tilde f}} \\
  &\lesssim A^{2d}\Theta+\eps_A C_R
\end{split}
\end{equation*}
for all $f\in L^\infty_0(Q_0)$ with $\Norm{f}{r'}\leq 1$ and $\Norm{f}{\infty}\leq R$, and thus
\begin{equation*}
  C_R\lesssim A^{2d}\Theta+\eps_A C_R.
\end{equation*}
Fixing $A$ large enough so that $\eps_A\ll 1$, we can absorb the last term and conclude that $C_R\lesssim \Theta$.

Let $L^\infty_{c,0}(\R^d):=\bigcup_{Q\subset\R^d}L^\infty_0(Q)$. Since every $f\in L^\infty_{c,0}(\R^d)$ satisfies $\Norm{f}{\infty}\leq R$ for some $R$, we conclude that
\begin{equation*}
  \Babs{\int bf}\lesssim\Theta\Norm{f}{r'}\qquad \forall f\in L^\infty_{c,0}(\R^d).
\end{equation*}
As this is a dense subspace of $L^{r'}(\R^d)$, there exists a unique bounded linear functional $\Lambda\in (L^{r'}(\R^d))^*$ such that
\begin{equation*}
  \Norm{\Lambda}{(L^{r'}(\R^d))^*}\lesssim \Theta,\qquad\Lambda(f)=\int bf\quad\forall f\in L^\infty_0(\R^d).
\end{equation*}
By the Riesz representation theorem, such a $\Lambda\in (L^{r'}(\R^d))^*$ is represented by a unique function $a\in L^r(\R^d)$ of the same norm, and hence
\begin{equation*}
  \Norm{a}{r}\lesssim\Theta,\qquad\int af=\int bf\quad\forall f\in L^\infty_0(\R^d).
\end{equation*}

Let $\Delta:=b-a\in L^1_{\loc}(\R^d)$. We have $\int \Delta\cdot f=0$ for all $f\in L^\infty_0(\R^d)$. Taking $f=t^{-d}(1_{B(x,t)}-1_{B(y,t)})$ and letting $t\to 0$, we deduce that $\Delta(x)=\Delta(y)$ for all Lebesgue points $x$ and $y$ of $\Delta$. Thus $\Delta(x)\equiv c$ is a constant, and $b=a+c$ with $\Norm{a}{r}\lesssim\Theta$, as claimed.
\end{proof}

\section{Applications to the Jacobian operator}\label{sec:Jacobi}

We now discuss applications of the previous methods towards the problem of finding an unknown function $u$ with the prescribed Jacobian
\begin{equation*}
  Ju=\operatorname{det}\nabla u=\det(\partial_i u_j)_{i,j=1}^d=f.
\end{equation*}
The Jacobian equation has been quite extensively studied in the form of a Dirichlet boundary value problem in a bounded, sufficiently smooth domain $\Omega\subset\R^d$,
\begin{equation}\label{eq:JBVP}
  \begin{cases} Ju=f & \text{in }\Omega, \\ \phantom{J}u=g & \text{on }\partial\Omega. \end{cases}
\end{equation}
There are several works dealing with datum $f$ in H\"older \cite{DM} or Sobolev spaces \cite{Ye:1994}; in a different direction, a recent result \cite[Theorem 6.3]{KRW} addresses $f\in L^p(\Omega)$ with $p\in(\frac1d,1)$.

Our interest is in the conjecture of Iwaniec \cite{Iwaniec:Escorial} discussed in Section \ref{ss:intro-Jacobi}; besides being set on the full space $\R^d$, it deals with datum $f$ in the spaces $L^p(\R^d)$, $p\in(1,\infty)$, which fall in some sense ``between'' the higher regularity classes considered by \cite{DM,Ye:1994}, and the sub-integrability classes in \cite{KRW}. The closest analogue of our results in the existing literature is the Hardy space $H^1(\R^d)$ results of Coifman, Lions, Meyer and Semmes \cite{CLMS}.

\subsection{Norming properties of Jacobians}
We prove that the norm of a function $b$ in various function spaces can be computed by dualising against functions in the range of the Jacobian operator. The following lemma, a variant of considerations used in \cite[p. 263]{CLMS}, already gives a flavour of such results:

\begin{lemma}\label{lem:Jacobi}
Let $b\in L^1_{\loc}(\R^d)$. For each $q\in(1,\infty)$ we have
\begin{equation*}
  \fint_Q\abs{b-\ave{b}_Q}
  \lesssim\sup\Big\{\Babs{\fint_{2Q} b J(u)}: u\in C_c^\infty(2Q)^d,   
  \fint_{2Q}\abs{\nabla u}^{q}\leq 1\Big\}.
\end{equation*}
\end{lemma}

\begin{proof}
We can find $g\in L^\infty_0(Q')$, supported in a slightly smaller cube $Q'=(1-\delta)Q$, and with $\Norm{g}{\infty}\leq 1$ such that
\begin{equation}\label{eq:easyDual}
   \fint_Q\abs{b-\ave{b}_Q}\lesssim\Babs{\fint_Q(b-\ave{b}_Q)g}=\Babs{\fint_Q bg}.
\end{equation}
Now $g\in L^q_0(Q')$ for every $q\in(1,\infty)$. By \cite[Lemma II.2.1.1]{Sohr:NSE}, we can find at least one $v\in W^{1,q}_0(Q')^d$ (Sobolev space with zero boundary values) satisfying
\begin{equation*}
  \operatorname{div}v=g,\qquad\Norm{\nabla v}{q}\lesssim \Norm{g}{q}. 
\end{equation*}
In fact, \cite[Lemma II.2.1.1]{Sohr:NSE} proves this with an unspecified dependence on the cube (or more generally, a Lipschitz domain) $Q'$; we apply this in the unit cube $Q_0$ first, and then obtain the stated estimate in an arbitrary cube by a change of variables.
So we have
\begin{equation*}
  \fint_Q\abs{b-\ave{b}_Q}\lesssim\Babs{\fint_Q b\operatorname{div}v},\qquad v\in W^{1,q}_0(Q')^d,\qquad \fint_Q\abs{\nabla v}^q\lesssim 1.
\end{equation*}
If we now replace $v$ by a standard mollification $\phi_\eps*v$ and note that $\nabla(\phi_\eps*v)=\phi_\eps*\nabla v$, we observe that the above display remains valid for small enough $\eps>0$, except that $v\in W^{1,q}_0(Q')^d$ is replaced by $\phi_\eps*v\in C_c^\infty(Q)^d$. We now proceed with this replacement, writing $w=\phi_\eps*v$.

Next, at least one of the integrals $\int_Q b\partial_k w_k$, $k=1,\ldots,d$, has to be at least as big as their average $d^{-1}\fint_Q b\operatorname{div}w$, so in fact
\begin{equation*}
  \fint_Q\abs{b-\ave{b}_Q}\lesssim\Babs{\fint_Q b\partial_k w_k},\qquad
  w_k\in C_c^\infty(Q),\qquad\fint_Q\abs{\nabla w_k}^q\lesssim 1.
\end{equation*}

We now define a vector-valued function $u=(u_i)_{i=1}^d\in C_c^\infty(2Q)$ as follows. For $i=k$, let $u_k=w_k$. For all $i\neq k$, let $u_i(x)=(x_i-c_i)\varphi_Q(x)$, where $c$ is the centre of $Q$, we write $x_i$ (resp. $c_i$) for the $i$th component of $x$ (resp. $c$), and $\varphi_Q\in C_c^\infty(2Q)$ is a usual bump such that $1_Q\leq\varphi_Q\leq 1_{2Q}$ and $\abs{\nabla\varphi_Q}\lesssim 1/\ell(Q)$. Then
\begin{equation*}
  \nabla u_i(x)=e_i\varphi_Q(x)+(x_i-c_i)\nabla\varphi_Q(x),\qquad
  \abs{\nabla u_i(x)}\lesssim 1_{2Q}(x),
\end{equation*}
where $e_i$ is the $i$th coordinate vector.
Thus $\fint_{2Q}\abs{\nabla u_i}^{q}\lesssim 1$ for $i\neq k$, and we already knew this for $i=k$. Since $u_k=w_k$ is compactly supported inside $Q$, so is $J(u)$, and for $x\in Q$, we simply have $\nabla u_i(x)=e_i$ for $i\neq k$. Hence
\begin{equation*}
  J(u)(x)=\det(\partial_i u_j(x))_{i,j=1}^d=\sum_{\sigma\in S_d}\sign(\sigma)\prod_{i\neq k}\delta_{i,\sigma(i)}\times\partial_{\sigma(k)}w_k(x)
  =\partial_k w_k(x),
\end{equation*}
since only the identity permutation gives a contribution. We have shown that
\begin{equation*}
   \fint_Q\abs{b-\ave{b}_Q}
  \lesssim\Babs{\fint_Q b\partial_k w_k}
  =\Babs{\fint_Q bJ(u)},
\end{equation*}
for a certain $u\in C_c^\infty(2Q)^d$ such that $\fint_{2Q}\abs{\nabla u}^{q}\lesssim 1$, and this proves the lemma.
\end{proof}

For the passage from the local estimate of Lemma \ref{lem:Jacobi} to global function space norms, we need two further lemmas that have nothing to do with the Jacobian, and will also be used in the next section.

\begin{lemma}\label{lem:Lerner}
Let $b\in L^1_{\loc}(\R^d)$ and let $Q_0\subset\R^d$ be a cube. Then there is collection $\mathcal Q$ of dyadic subcubes of $Q_0$ such that, at almost every $x\in Q_0$,
\begin{equation*}
  1_{Q_0}(x)\abs{b-\ave{b}_{Q_0}}\lesssim\sum_{Q\in\mathcal Q} 1_Q(x)\fint_Q\abs{b-\ave{b}_Q},
\end{equation*}
and $\mathcal Q$ is sparse in the sense that each $Q\in\mathcal Q$ has a major subset $E(Q)$ such that $\abs{E(Q)}\geq\frac12\abs{Q}$ and the subsets $E(Q)$ are pairwise disjoint.
\end{lemma}

\begin{proof}
This is a more elementary variant of Lerner's oscillation formula \cite{Lerner:formula}; we recall the idea of the proof. For any disjoint subcubes $Q^1_j$ of $Q_0$, we have
\begin{equation}\label{eq:basicSplit}
\begin{split}
  1_{Q_0}(b-\ave{b}_{Q_0})
  &=1_{Q_0\setminus\bigcup_j Q^1_j}(b-\ave{b}_{Q_0}) \\
  &\qquad+\sum_j 1_{Q^1_j}(\ave{b}_{Q^1_j}-\ave{b}_{Q_0})+\sum_j 1_{Q^1_j}(b-\ave{b}_{Q^1_j}).
\end{split}
\end{equation}
If the $Q^1_j$ are chosen to be the maximal dyadic subcubes $Q\subset Q_0$ such that
\begin{equation*}
  \fint_{Q}\abs{b-\ave{b}_{Q_0}}>2\fint_{Q_0}\abs{b-\ave{b}_{Q_0}},
\end{equation*}
then $\sum_j\abs{Q^1_j}\leq\frac12\abs{Q_0}$ so that $E(Q_0)=Q_0\setminus\bigcup_j Q^1_j$ qualifies for a major subset. Moreover, the sum of the first two terms on the right of \eqref{eq:basicSplit} is dominated by $1_{Q_0}\fint_{Q_0}\abs{b-\ave{b}_{Q_0}}$ and the last term is a sum over disjointly supported terms of the same form as where we started, and we can iterate.
\end{proof}

We borrow the following observation from \cite[Remark 2.4]{DHKY}:

\begin{lemma}\label{lem:loc2glob}
Suppose that $b\in L^r_{\loc}(\R^d)$, $r\in[1,\infty)$, satisfies
\begin{equation*}
  \Norm{b-\ave{b}_Q}{L^r(Q)}\leq \Theta
\end{equation*}
for every cube $Q\subset\R^d$. Then $b=a+c$, where $c$ is a constant, $a\in L^r(\R^d)$, and
\begin{equation*}
  \Norm{a}{L^r(\R^d)}\leq \Theta.
\end{equation*}
\end{lemma}

\begin{proof}
Let us consider a sequence of cubes $Q_0\subset Q_1\subset\ldots$ with $\bigcup_{n=0}^\infty Q_n=\R^d$. For $m\leq n$, we have
\begin{equation*}
  \abs{\ave{b}_{Q_n}-\ave{b}_{Q_m}}
  =\abs{(b(x)-\ave{b}_{Q_m})-(b(x)-\ave{b}_{Q_n})}
\end{equation*}
and hence, taking the $L^r$ average over $x\in Q_m$,
\begin{equation*}
\begin{split}
  \abs{\ave{b}_{Q_n}-\ave{b}_{Q_m}}
  &\leq\abs{Q_m}^{-1/r}(\Norm{b-\ave{b}_{Q_m}}{L^r(Q_m)}+\Norm{b-\ave{b}_{Q_n}}{L^r(Q_m)}) \\ 
  &\leq \abs{Q_m}^{-1/r}(\Theta+\Norm{b-\ave{b}_{Q_n}}{L^r(Q_n)}) \leq 2\Theta\abs{Q_m}^{-1/r}.
\end{split}
\end{equation*}
Thus $(\ave{b}_{Q_n})_{n=0}^\infty$ is a Cauchy sequence, and hence converges to some $c$. We conclude by Fatou's lemma that
\begin{equation*}
  \int_{\R^d}\abs{b-c}^r
  =\int_{\R^d}\lim_{n\to\infty}1_{Q_n}\abs{b-\ave{b}_{Q_n}}^r
  \leq\liminf_{n\to\infty}\int_{Q_n}\abs{b-\ave{b}_{Q_n}}^r
  \leq \Theta^r.\qedhere
\end{equation*}
\end{proof}

We are now ready for the main result of this section:

\begin{theorem}\label{thm:equiNormJu}
Let $b\in L^1_{\loc}(\R^d)$, let $r_i\in(1,\infty)$ for $i=1,\ldots,d$, and $\displaystyle \frac1r:=\sum_{i=1}^d\frac{1}{r_i}$.
Then
\begin{equation}\label{eq:supJu}
  \Gamma:=\sup\Big\{\Babs{\int b J(u)}:u=(u_i)_{i=1}^d\in C_c^\infty(\R^d)^d, \Norm{\nabla u_i}{r_i}\leq 1\ \forall i=1,\ldots,d\Big\}
\end{equation}
is finite, if and only if
\begin{itemize}
  \item $r=1$ and $b\in\BMO(\R^d)$, or
  \item $r\in [\frac{d}{d+1},1)$ and $b$ is $d(\frac1r-1)$-H\"older continuous, or
  \item $r<\frac{d}{d+1}$ and $b$ is constant, or
  \item $r>1$ and $b=a+c$, where $c$ is constant and $a\in L^{r'}(\R^d)$.
\end{itemize}
Moreover, in each case the respective function space norm is comparable to \eqref{eq:supJu}.
\end{theorem}

\begin{proof}
Let us first consider the ``if'' directions. The constant cases follow from the fact that $\int J(u)=0$, and it is immediate from H\"older's inequality that
\begin{equation*}
  \Babs{\int bJ(u)}\lesssim\Norm{b}{r'}\prod_{i=1}^d\Norm{\nabla u_i}{r_i},\qquad r>1.
\end{equation*}

We then deal with $r\in[\frac{d}{d+1},1]$. Let us first check that there is at least one $k$ such that $1/r-1/r_k<1$. Suppose for contradiction that we have $1/r-1/r_k\geq 1$ for all $k=1,\ldots,d$. Summing over $k$, this gives $d/r-1/r\geq d$, and thus $r\leq \frac{d-1}{d}$. But we are also assuming that $\frac{d}{d+1}\leq r$, thus $d^2\leq(d-1)(d+1)=d^2-1$, a contradiction. Without loss of generality, we may assume that $k=1$, thus
\begin{equation*}
  \frac{1}{s}:=\sum_{i=2}^d\frac{1}{r_i}=\frac{1}{r}-\frac{1}{r_1}\in(0,1)
\end{equation*}
so that $s\in(1,\infty)$. We can then write
\begin{equation*}
  J(u)=\nabla u_1\cdot\sigma= Rf\cdot\sigma,
\end{equation*}
where $\sigma\in C_c^\infty(\R^d)^d$ satisfies $\operatorname{div}\sigma =0$ and
\begin{equation*}
  \Norm{\sigma}{s}\lesssim\prod_{i=2}^d\Norm{\nabla u_i}{r_i}\leq 1, 
\end{equation*}
and $R=(R_i)_{i=1}^d=\nabla(-\Delta)^{-1/2}$ is the vector of the Riesz transforms, and finally $f=(-\Delta)^{1/2}u_1$ satisfies $\Norm{f}{r_1}\lesssim\Norm{\nabla u_1}{r_1}\leq 1$. Then
\begin{equation*}
  -\int R(bf)\cdot\sigma
  =\int bf (R\cdot \sigma)=\int bf(-\Delta)^{-1/2}\operatorname{div}\sigma=0,
\end{equation*}
and thus
\begin{equation*}
\begin{split}
  \Babs{\int bJ(u)}
  =\Babs{\int b Rf\cdot\sigma}
  =\Babs{\int [b,R]f\cdot\sigma}
  \leq\Norm{[b,R]}{L^{r_1}\to L^{s'}}\Norm{f}{r_1}\Norm{\sigma}{s}.
\end{split}
\end{equation*}
The last two norms are bounded by one, and $1/s'=1-1/r+1/r_1$, so that
\begin{equation*}
  \Gamma\leq\Norm{[b,R]}{L^{r_1}\to L^{s'}}\lesssim\begin{cases} \Norm{b}{\BMO}, & \text{if } r=1, \\ \Norm{b}{\dot C^{0,\alpha}}, & \text{if }\alpha:=d(\frac{1}{r_1}-\frac{1}{s'})=d(\frac{1}{r}-1)\in(0,1] \end{cases}
\end{equation*}
by Theorem \ref{thm:summary}.

We turn to the ``only if'' parts of the theorem. Recall the definition of $\Gamma$ from \eqref{eq:supJu}.
We apply Lemma \ref{lem:Jacobi} with some $q>\max_{i=1,\ldots,d}r_i$. If $Q\subset\R^d$ is any cube, then for some $u\in C_c^\infty(2Q)^d$ with $\fint_{2Q}\abs{\nabla u}^q\leq 1$ we have
\begin{equation}\label{eq:JuStep}
\begin{split}
   \fint_Q\abs{b-\ave{b}_Q}&\lesssim\Babs{\fint_{2Q} b J(u)}
  \lesssim\frac{\Gamma}{\abs{Q}}\prod_{i=1}^d\Norm{\nabla u_i}{r_i} \\
  &\lesssim \Gamma\frac{\abs{Q}^{1/r}}{\abs{Q}}\prod_{i=1}^d\Big(\fint_{2Q}\abs{\nabla u_i}^{r_i}\Big)^{1/r_i} \\
  &\lesssim \Gamma\ell(Q)^{d(1/r-1)}\prod_{i=1}^d\Big(\fint_{2Q}\abs{\nabla u_i}^{q}\Big)^{1/q}
  \leq \Gamma\ell(Q)^{d(1/r-1)}.
  \end{split}
\end{equation}
If $r=1$, this is precisely the condition that $\Norm{b}{\BMO}\lesssim \Gamma$. For $r<1$, the conclusion follows as in the proof of Theorem \ref{thm:pleq1}.

Let us then consider $r>1$. Let $Q_0\subset\R^d$ be an arbitrary cube. We apply Lemma \ref{lem:Lerner} and monotone convergence to see that
\begin{equation*}
  \Norm{b-\ave{b}_{Q_0}}{L^{r'}(Q_0)}
  \lesssim\lim_{N\to\infty}\BNorm{\sum_{k=1}^N 1_{Q_k}\fint_{Q_k}\abs{b-\ave{b}_{Q_k}}}{L^{r'}(Q_0)}
\end{equation*}
where $\{Q_k\}_{k=1}^\infty$ is an enumeration of the collection $\mathcal Q$ given by Lemma \ref{lem:Lerner}.

We then dualise with some $\Norm{\phi}{r}\leq 1$, and apply  just the first step of \eqref{eq:JuStep} to each $Q_k$ in place of $Q$. 
Note that this produces a possibly different $u^k=(u^k_i)_{i=1}^d\in C_c^\infty(2Q_k)^d$ for each $k$. Thus we end up estimating
\begin{equation*}
\begin{split}
  \sum_{k=1}^N \int_{Q_k}\phi \ \fint_{Q_k}\abs{b-\ave{b}_{Q_k}}
  &\lesssim\sum_{k=1}^N\int_{Q_k}\phi\ \fint_{2Q_k} b J(u^k) \\
  &\lesssim\int b\sum_{k=1}^N J(u^k)\lambda_k,\qquad\lambda_k:=\fint_{Q_k}\phi.
\end{split}
\end{equation*}
In order to proceed, we make a randomisation trick. Due to the $d$-linear nature of the Jacobian, we invoke a sequence $(\zeta_k)_{k=1}^N$ of independent random $d$th roots of unity, i.e. the $\zeta_k$'s are independent random variables on some probability space, distributed so that $\mathbb P(\zeta_k=e^{i2\pi a/d})=1/d$ for each $a=0,1,\ldots,d-1$. The case $d=2$ thus corresponds to the familiar random signs. The important feature of these random variables is that, denoting by $\Exp$ the expectation,
\begin{equation*}\tag{$*$}
  \Exp\prod_{j=1}^d\zeta_{k_j}=\begin{cases} 1, & \text{if }k_1=\ldots=k_d, \\ 0, & \text{else}.\end{cases}
\end{equation*}
Indeed, if $k_1=\ldots=k_d=k$, then $\prod_{j=1}^d\zeta_{k_j}=\zeta_k^d\equiv 1$, so also its expectation is equal to $1$. Otherwise, we have $\prod_{j=1}^d\zeta_{k_j}=\prod_{j=1}^r\zeta_{m_j}^{n_j}$ for some {\em distinct} values $m_1,\ldots,m_r\in\{1,\ldots,N\}$ and exponents $n_1,\ldots,n_r\in\{1,\ldots,d-1\}$. By independence, it then follows that
\begin{equation*}
   \Exp\prod_{j=1}^d\zeta_{k_j}=\Exp\prod_{j=1}^r\zeta_{m_j}^{n_j}
   =\prod_{j=1}^r\Exp\zeta_{m_j}^{n_j},
\end{equation*}
where each factor satisfies
\begin{equation*}
  \Exp\zeta_{m_j}^{n_j}
  =\frac{1}{d}\sum_{a=0}^{d-1}e^{i2\pi a n_j/d}=0,
\end{equation*}
noting that $e^{i2\pi n_j/d}\neq 1$ since $0<n_j<d$.

Using $(*)$, we can now continue the computation from above with
\begin{equation*}
\begin{split}  
 \int b\sum_{k=1}^N J(u^k)\lambda_k
  &\overset{(*)}{=}\Exp\int bJ(\sum_{k_1=1}^N\eps_{k_1}\lambda_{k_1}^{r/r_1}u^{k_1}_1,\cdots,\sum_{k_d=1}^N\eps_{k_d}\lambda_{k_d}^{r/r_d}u^{k_d}_d) \\
  &\leq \Gamma\Exp\prod_{i=1}^d\BNorm{\sum_{k=1}^N\eps_k\lambda_k^{r/r_i}\nabla u^k_i}{r_i} \\
  &\leq \Gamma\prod_{i=1}^d\BNorm{\sum_{k=1}^N\lambda_k^{r/r_i}\abs{\nabla u^k_i}}{r_i}.
\end{split}
\end{equation*}

To estimate each $L^{r_i}$ norm above, we dualise with $\Norm{\psi}{r_i'}\leq 1$. Recalling that $u_i^k\in C_c^\infty(2Q_k)$ satisfies the bound for $u$ in Lemma \ref{lem:Jacobi}, and using the definition of $\lambda_k$ and the disjoint major subsets $E(Q_k)$ from Lemma \ref{lem:Lerner}, we have
\begin{equation*}
\begin{split}
  \int \psi\sum_{k=1}^N\lambda_k^{r/r_i}\abs{\nabla u^k_i}
  &=\sum_{k=1}^N\lambda_k^{r/r_i}\int_{2Q_k}\abs{\nabla u^k_i}\psi \\
  &\lesssim\sum_{k=1}^N\lambda_k^{r/r_i}\abs{Q_k}\Big(\fint_{2Q_k}\abs{\nabla u^k_i}^q\Big)^{1/q}\Big(\fint_{2Q_k}\psi^{q'}\Big)^{1/q'} \\
  &\leq\sum_{k=1}^N\Big(\fint_{Q_k}\phi\Big)^{r/r_i}\abs{Q_k}\Big(\fint_{2Q_k}\psi^{q'}\Big)^{1/q'} \\
  &\lesssim\sum_{k=1}^N \abs{E(Q_k)} \Big(\inf_{x\in Q_k}M\phi(x)\Big)^{r/r_i}\Big(\inf_{y\in Q_k}M(\psi^{q'})(y)\Big)^{1/q'} \\
  &\leq\int (M\phi)^{r/r_i} (M(\psi^{q'}))^{1/q'} \leq\Norm{(M\phi)^{r/r_i}}{r_i}\Norm{(M(\psi^{q'}))^{1/q'}}{r_i'} \\
  &=\Norm{M\phi}{r}^{r/r_i}\Norm{M(\psi^{q'})}{r_i'/q'}^{1/q'}\lesssim\Norm{\phi}{r}^{r/r_i}\Norm{\psi}{r_i'}\leq 1,
\end{split}
\end{equation*}
by the boundedness of the maximal operator and the choice of $q>r_i$ so that $r_i'>q'$.

Substituting back, we have checked that
\begin{equation*}
  \Norm{b-\ave{b}_{Q_0}}{L^{r'}(Q_0)}\lesssim \Gamma
\end{equation*}
for an arbitrary cube $Q_0$; by Lemma \ref{lem:loc2glob}, this completes the proof of the theorem in the remaining case that $r>1$.
\end{proof}

\begin{remark}
The ``if'' parts of the cases $r\in(\frac{d}{d+1},1]$ of Theorem \ref{thm:equiNormJu} could also be deduced from a result of \cite[cf. Theorem II.3]{CLMS}, which says that $J(u)$ belongs to the Hardy space $H^r(\R^d)$ under the same assumptions, together with the $H^1$-$\BMO$ duality when $r=1$ or the $H^r$-$\dot C^{0,d(1/r-1)}$-duality for $r\in(\frac{d}{d+1},1)$. However, a separate argument would be required for the end-point $r=\frac{d}{d+1}$ any way: in fact, $J:C_c^\infty(\R^d)\not\to H^{d/(d+1)}(\R^d)$, since $J(u)$ fails, in general, to satisfy the required moment conditions $\int x_i a=0$ of an $H^{d/(d+1)}$-atom $a$. This follows e.g. from the proof of Lemma \ref{lem:Jacobi}, which contains the observation that any $\partial_k w$, with $w\in C_c^\infty(\R^d)$, can arise as the Jacobian $J(u)$ of a suitable $u\in C_c^\infty(\R^d)^d$. However, we have $\int x_k\partial_k w=-\int w\partial_k x_k=-\int w$, which can easily be nonzero. The departure from the Hardy-H\"older duality is also reflected by the fact that the condition for $b$ in Theorem \ref{thm:equiNormJu} corresponding to $r=\frac{d}{d+1}$ is the usual Lipschitz-continuity, $\abs{b(x)-b(y)}\lesssim\abs{x-y}$, and not the Zygmund class condition arising from the Hardy space duality.

On the other hand, one can also give a different proof of the ``if'' part of Theorem \ref{thm:equiNormJu} in this special case $r=\frac{d}{d+1}$. Using the notation from the previous proof, where $1>\frac1s=\frac1r-\frac{1}{r_1}=1+\frac1d-\frac{1}{r_1}$, we find that $r_1\in(1,d)$. Writing, as before, $J(u)=\nabla u_1\cdot\sigma$, we have
\begin{equation*}
  \int bJ(u)
  =\int b \nabla u_1\cdot\sigma
  =-\int u_1\operatorname{div}(b\sigma)
  =-\int u_1(\nabla b)\cdot\sigma,
\end{equation*}
since $\operatorname{div}\sigma=0$. But then we can estimate
\begin{equation*}
  \Babs{\int u_1(\nabla b)\cdot\sigma}
  \leq\Norm{u_1}{s'}\Norm{\nabla b}{\infty}\Norm{\sigma}{s},
\end{equation*}
where $\Norm{\nabla b}{\infty}$ is bounded by the Lipschitz constant, $\Norm{\sigma}{s}\leq 1$, and
\begin{equation*}
  \frac{1}{s'}=1-\frac{1}{s}=1-\Big(\frac{1}{r}-\frac{1}{r_1}\Big)=\frac{1}{r_1}-\frac{1}{d},
\end{equation*}
so that that $s'=r_1 d/(d-r_1)=r_1^*$ is the Sobolev exponent. Thus
\begin{equation*}
  \Norm{u_1}{s'}=\Norm{u_1}{r_1^*}\lesssim\Norm{\nabla u_1}{r_1}\leq 1.
\end{equation*}
by Sobolev's inequality, and this completes the alternative proof.
\end{remark}

\subsection{The linear span of Jacobians}

Here we will obtain the following consequence of Theorem \ref{thm:equiNormJu}:

\begin{theorem}\label{thm:spanJ}
Let $d\geq 2$ and $p\in[1,\infty)$. Then
\begin{equation*}
\begin{split}
  \Big\{\sum_{j=1}^\infty Ju^j: u^j\in\dot W^{1,pd}(\R^d)^d,\ \sum_{j=1}^\infty\ & \Norm{\nabla u^j}{L^{pd}(\R^d)^{d\times d}}^d< \infty \Big\} \\
  &=\begin{cases}
     L^p(\R^d), & p\in(1,\infty), \\
     H^1(\R^d), & p=1.
    \end{cases}
\end{split}
\end{equation*}
In fact, each $f\in L^p(\R^d)$ (resp. $f\in H^1(\R^d)$) admits a representation
\begin{equation*}
  f=\sum_{j=1}^\infty Ju^j,\quad\sum_{j=1}^\infty\Norm{\nabla u^j}{L^{pd}(\R^d)^{d\times d}}^d\lesssim
  \begin{cases}
     \Norm{f}{L^p(\R^d)}, & p\in(1,\infty), \\
     \Norm{f}{H^1(\R^d)}, & p=1,
    \end{cases}
\end{equation*}
where each $u^j\in C_c^\infty(\R^d)^d$.
\end{theorem}

The power $d$ in the series is related to the $d$-homogeneity of the Jacobian, so that $\Norm{\nabla u^j}{L^{pd}(\R^d)^{d\times d}}^d$ is (up to constant) an upper bound for $\Norm{Ju^j}{L^p(\R^d)}$ or $\Norm{Ju^j}{H^1(\R^d)}$ for $p=1$. The case $p=1$ is already due to Coifman et al.~\cite{CLMS}; they explicitly formulate a similar result \cite[Theorem III.2]{CLMS} for the ``div-curl example'' but point out that ``this type of answer applies also to other examples like the Jacobian''. Our proof of the full Theorem \ref{thm:spanJ} depends on the same functional analytic lemma as used in \cite{CLMS} for the case $p=1$. The formulation below combines \cite[Lemmas III.1, III.2]{CLMS} and is taken from \cite{Lindberg:2017}. We recall the short proof for the sake of recording a precise quantitative relation between the equivalent qualitative conditions:

\begin{lemma}\label{lem:CLMS}
Let $V\subset\bar B_X(0,1)$ be a symmetric subset of the unit-ball of a  Banach space $X$. Then the following conditions are equivalent:
\begin{enumerate}
  \item\label{it:equiNorms} There is $\alpha>0$ such that $\sup_{x\in V}\abs{\pair{\lambda}{x}}\geq\alpha\Norm{\lambda}{X^*}$ for all $\lambda\in X^*$.
  \item\label{it:clConv} The closed convex hull $\overline{\operatorname{conv}}(V)$ contains a ball $\bar B_X(0,\beta)$ of radius $\beta>0$.
  \item\label{it:sConv} The {\em $s$-convex hull}
\begin{equation*}
  s(V):=\Big\{\sum_{j=1}^\infty \lambda_j x_j: x_j\in V,\lambda_j\geq 0,\sum_{j=1}^\infty\lambda_j=1\Big\}
\end{equation*}
contains an open ball $B_X(0,\gamma)$ of radius $\gamma>0$.
\end{enumerate}
Moreover, the largest admissible values of $\alpha,\beta,\gamma$ satisfy $\alpha=\beta=\gamma$.
\end{lemma}

\begin{proof}
If $\lambda\in X^*$, we can find $x_0\in \bar B_X(0,\beta)$ such that $\beta\Norm{\lambda}{X^*}=\abs{\pair{\lambda}{x_0}}$. Writing this $x_0$ as $x_0=\lim_n x_n$, where $x_n\in\operatorname{conv}(V)$, we easily check that $\sup_{x\in V}\abs{\pair{\lambda}{x}}\geq\beta\Norm{\lambda}{X^*}$, and hence $\alpha\geq\beta$. On the other hand, if $y_0\notin\overline{\operatorname{conv}}(V)$, then by the Hahn--Banach theorem there exists $\lambda\in X^*$ such that $\Re\pair{\lambda}{x}\leq\eta<\Re\pair{\lambda}{y_0}$ for some $\eta\in\R$ and all $x\in\overline{\operatorname{conv}}(V)$, in particular for $x\in V$, and thus, by the symmetry of $V$, also $\abs{\pair{\lambda}{x}}\leq\eta<\abs{\pair{\lambda}{y_0}}\leq\Norm{\lambda}{X^*}\Norm{y_0}{X}$ for all $x\in V$. Taking the supremum over $x\in V$ and using \eqref{it:equiNorms} it follows that $\alpha\Norm{\lambda}{X^*}\leq\eta<\Norm{\lambda}{X^*}\Norm{y_0}{X}$. Since clearly $\lambda\neq 0$, it follows that $\Norm{y_0}{X}>\alpha$, and thus $\beta\geq\alpha$.

Clearly $s(V)\subset\overline{\operatorname{conv}}(V)$, and hence $B_X(0,\gamma)\subset s(V)$ implies $\bar B_X(0,\gamma)\subset\overline{\operatorname{conv}}(V)$ so that $\beta\geq\gamma$. On the other hand, suppose that $x\in\bar B(0,\beta)\subset\overline{\operatorname{conv}}(V)$. Fix $\eps>0$. Suppose that we have already found $x_k\in\operatorname{conv}(V)$ such that
\begin{equation}\label{eq:FAind}
   \Norm{x-\sum_{k=0}^{n-1}\eps^k x_k}{X}\leq \eps^n\beta
\end{equation}
(this is vacuous for $n=0$). Then $\eps^{-n}(x-\sum_{k=0}^{n-1}\eps^k x_k)\in\bar B_X(0,\beta)\subset\overline{\operatorname{conv}}(V)$, and thus we can pick $x_n\in\operatorname{conv}(V)$ with $\Norm{\eps^{-n}(x-\sum_{k=0}^{n-1}\eps^k x_k)-x_n}{X}\leq\eps\beta$. But this is the same as \eqref{eq:FAind} with $n+1$ in place of $n$. By induction it follows that $x=\sum_{k=0}^\infty\eps_k x_k$ with $x_k\in\operatorname{conv}(V)$. Since $\sum_{k=0}^\infty\eps^k=(1-\eps)^{-1}$, this means that $(1-\eps)x\in s(V)$. As $x\in\bar B_X(0,\beta)$ and $\eps>0$ were arbitrary, we have $\bar B_X(0,(1-\eps)\beta)\subset s(V)$, and hence $B_X(0,\beta)\subset s(V)$. Thus $\gamma\geq\beta$.
\end{proof}

\begin{proof}[Proof of Theorem \ref{thm:spanJ}]
We apply Lemma \ref{lem:CLMS} with $X=L^p(\R^d)$ if $p\in(1,\infty)$, or $X=H^1(\R^d)$ if $p=1$. In either case, let
\begin{equation*}
  V=\{Ju:u\in C_c^\infty(\R^d)^d,\Norm{\nabla u}{L^{pd}(\R^d)^{d\times d}}\leq 1\}.
\end{equation*}
It is immediate that $V$ is symmetric, and that $V\subset\bar B_X(0,1)$ if $p>1$. For $p=1$, this last inclusion is nontrivial but well known from \cite[Theorem II.1]{CLMS}.

The assertion of Theorem \ref{thm:spanJ} is clearly the same as \eqref{it:sConv} of Lemma \ref{lem:CLMS} for these choices of $X$ and $V$. By Lemma \ref{lem:CLMS}, it hence suffices to verify \eqref{it:equiNorms} of the same lemma, i.e., that 
\begin{equation*}
\begin{split}
   \Norm{b}{X^*}\lesssim\sup\Big\{\Babs{\int bJ(u)} &: u\in C^\infty_c(\R^d)^d:\Norm{\nabla u}{L^p(\R^d)^{d\times d}}\leq 1\Big\},\\
   &\forall b\in X^*=\begin{cases} L^{p'}(\R^d), & p\in(1,\infty), \\ \BMO(\R^d), & p=1. \end{cases}
\end{split}
\end{equation*}
But this is precisely the statement of Theorem \ref{thm:equiNormJu} for $r=p\in[1,\infty)$ and $r_1=\ldots=r_d=pd$. The a priori condition that $b\in L^{p'}(\R^d)$ guarantees that the additive constant present in Theorem \ref{thm:equiNormJu} for $r>1$ does not appear here.
\end{proof}

\begin{remark}
\begin{enumerate}
  \item Lindberg \cite[Lemma 3.1]{Lindberg:2017} shows that another equivalent condition in Lemma \ref{lem:CLMS} is that $\bigcup_{n=1}^\infty n\cdot s(V)$ has second category in $X$. Hence, if any of these conditions fails, then $\bigcup_{n=1}^\infty n\cdot s(V)$ has first category in $X$. Lindberg uses this to show \cite[Theorems 1.2, 7.4]{Lindberg:2017} that the set
\begin{equation*}
  \Big\{\sum_{j=1}^\infty Ju^j:u^j\in W^{1,pd}(\R^d)^d,\sum_{j=1}^\infty\big(\Norm{u^j}{L^{pd}(\R^d)^d}+\Norm{\nabla u^j}{L^{pd}(\R^d)^{d\times d}}\big)^d<\infty\Big\}
\end{equation*}
has first category in $L^p(\R^d)$ if $p\in(1,\infty)$, or in $H^1(\R^d)$ if $p=1$.
  \item Lindberg \cite[p. 739]{Lindberg:2017} also sketches how to deduce the special case $d=2$ of Theorems \ref{thm:equiNormJu} and \ref{thm:spanJ} from the special case of (then unknown) Theorem \ref{thm:summary}, where $T$ is the Ahlfors--Beurling operator. Since a more general result is proved above by working directly with the Jacobian, we do not repeat his argument here, but the interested reader may consult the companion paper \cite{Hyt:Ricci} for this approach. Nevertheless, the strategy proposed by Lindberg was an important motivation for the discovery of our present results.
\end{enumerate}
\end{remark}

\section{Higher order real commutators and the median method}

In this section we establish the following variant of Theorem \ref{thm:summary}. In one direction, it generalises Theorem \ref{thm:summary} by allowing iterated commutators of arbitrary order, but in another direction it imposes a more restrictive assumption by requiring the pointwise multiplier $b$ to be real-valued. This restriction arises from the proof using the so-called median method, which takes explicit advantage of the order structure of the real line. We note, however, that this restriction is imposed on $b$ only; the kernel $K$ of $T$ may still be complex-valued.

\begin{theorem}\label{thm:higher}
Let $1<p,q<\infty$, let $T$ be a non-degenerate Calder\'on--Zygmund operator on $\R^d$, and let $k\in\{1,2,\ldots\}$ and $b\in L^k_{\loc}(\R^d;\R)$. Then the $k$ times iterated commutator
\begin{equation*}
  T_b^k:=[b,T_b^{k-1}],\qquad T_b^1:= [b,T],
\end{equation*}
defines a bounded operator $T_b^k:L^p(\R^d)\to L^q(\R^d)$ if and only if:
\begin{itemize}
  \item $p=q$ and $b$ has bounded mean oscillation, or 
  \item $\displaystyle p<q\leq p_k^* =\Big(\frac{1}{p}-\frac{k}{d}\Big)_+^{-1}$
  and $b$ is $\displaystyle \alpha=\frac{d}{k}\Big(\frac{1}{p}-\frac{1}{q}\Big)$-H\"older continuous, or
  \item $q>p_k^*$ and $b$ is constant, or
  \item $p>q$ and $b=a+c$, where $a\in L^{rk}(\R^d)$ for $\displaystyle \frac1r=\frac1q-\frac1p$, and $c$ is constant.
\end{itemize}
\end{theorem}

As in the case of Theorem \ref{thm:summary}, all the ``if'' statements are either classical (such as the case $p=q$ that goes back to Coifman, Rochberg and Weiss \cite{CRW}) or straightforward; this applies to the remaining cases, which may be handled by easy extensions of the arguments sketched for $k=1$ in Section \ref{ss:intro-suff}.
(There is also a variant of the $p<q$ case of Theorem \ref{thm:higher} due to Paluszy\'nski, Taibleson and Weiss \cite{PTW}, but for $k>1$, it deals with operators that are related to, but not exactly the same as, the iterated commutators $T_b^k$ that we study. This leads to a slightly different result.)

 As before, our principal task is to prove the ``only if'' directions.

\subsection{Basic estimates of the median method}

We will not give a formal definition of the ``median method'', but the reason for this nomenclature should be fairly apparent from the considerations that follow. The broad philosophy of this method should be attributed to Lerner, Ombrosi and Rivera-R{\'{\i}}os \cite{LORR}, but we fine-tune some of its details in such a way as to be able, in particular, to answer a problem that was raised but left open in \cite[Remark 4.1]{LORR}.

The simplest form of the median method is contained in the following lemma. Under a quantitative positivity assumption on the kernel (which may nevertheless be complex-valued!), it needs no additional ``Calder\'on--Zygmund'' structure. 

\begin{lemma}\label{lem:quposBalls}
Let $b\in L^k_{\loc}(\R^d;\R)$.
Suppose that, for some disjoint balls $B,\tilde B$ of equal radius $r$, 
we have
\begin{equation}\label{eq:quposBalls}
  \Re(\sigma K(x,y))\gtrsim \frac{1}{\abs{B}}\qquad\text{for all }x\in\tilde B, y\in B
\end{equation}
for some $\abs{\sigma}=1$. If $T$ has kernel $K$, then
\begin{equation*}
  \inf_c\int_B\abs{b(y)-c}^k\ud y\lesssim\sum_{i=1}^2\abs{\pair{1_{\tilde E_i}}{T_b^k 1_{E_i}}}
\end{equation*}
for some subsets $E_i\subset B$, $\tilde E_i\subset\tilde B$.
\end{lemma}

\begin{remark}\label{rem:quposBalls}
If $K$ is a non-degenerate two-variable Calder\'on--Zygmund kernel (Definition \ref{def:ndCZK}\eqref{it:ndVar}), then for all large enough $A$ and for every ball $B=B(y_0,r)$ there exists another ball $\tilde B=B(x_0,r)$ with $\abs{x_0-y_0}\eqsim Ar$, where the assumptions, and hence the conclusions, of Lemma \ref{lem:quposBalls} are satisfied.

Indeed, by Proposition \ref{prop:unifNonDegImplies} and case \eqref{it:ndVar} of its proof, we can find an $x_0$ with $\abs{x_0-y_0}\eqsim Ar$ such that
\begin{equation*}
  \abs{K(x_0,y_0)}\eqsim\frac{1}{(Ar)^d},\qquad\abs{K(x,y)-K(x_0,y_0)}\leq\frac{\eps_A}{(Ar)^d}\quad\forall x\in B(x_0,r),\ \forall y\in B(y_0,r).
\end{equation*}
Hence, for suitable $\sigma$, we have
\begin{equation*}
  \Re(\sigma K(x,y))
  \geq\abs{K(x_0,y_0)}-\abs{K(x,y)-K(x_0,y_0)}\eqsim\frac{1}{(Ar)^d}.
\end{equation*}
\end{remark}

\begin{proof}[Proof of Lemma \ref{lem:quposBalls}]
The basic observation is that, if $\alpha\in\R$ and $x\in\tilde B\cap\{ b\leq\alpha\}$, then
\begin{equation*}
\begin{split}
  \int_B (b(y)-\alpha)_+^k\ud y
  &\leq \int_{B\cap\{b\geq\alpha\}} (b(y)-b(x))^k\ud y \\
  &\lesssim (-1)^k \abs{B} \Re\Big(\sigma \int_{B\cap\{b\geq\alpha\}} (b(x)-b(y))^k K(x,y)\ud y\Big),
\end{split}
\end{equation*}
and hence
\begin{equation*}
\begin{split}
  &\abs{\tilde B\cap\{ b\leq\alpha\}} \int_B (b(y)-\alpha)_+^k\ud y \\
  &\lesssim \abs{B}\Babs{\int_{\tilde B\cap\{b\leq\alpha\}}\int_{B\cap\{b\geq\alpha\}} (b(x)-b(y))^k K(x,y)\ud y\ud x} \\
  &= \abs{B}\Babs{\pair{1_{\tilde B\cap\{b\leq\alpha\}}}{T_b^k 1_{B\cap\{b\geq\alpha\}}}}
    =: \abs{B}\cdot\abs{\pair{1_{\tilde E_1}}{T_b^k 1_{E_1}}}.
\end{split}
\end{equation*}

In a completely analogous way, integrating over $x\in\tilde B\cap \{ b\geq\alpha\}$, we also prove that
\begin{equation*}
  \abs{\tilde B\cap\{ b\geq\alpha\}} \int_B (\alpha-b(y))_+^k\ud y
  \lesssim \abs{B}\Babs{\pair{1_{\tilde B\cap\{b\geq\alpha\}}}{T_b^k 1_{B\cap\{b\leq\alpha\}}}}
  = \abs{B}\cdot\abs{\pair{1_{\tilde E_2}}{T_b^k 1_{E_2}}}.
\end{equation*}
Choosing $\alpha$ as a median of $b$ on $\tilde B$, we have
\begin{equation*}
  \min(\abs{\tilde B\cap\{ b\leq\alpha\}},\abs{\tilde B\cap\{ b\geq\alpha\}})\geq\frac12\abs{\tilde B}=\frac12\abs{B},
\end{equation*}
and hence
\begin{equation*}
\begin{split}
  \int_B\abs{b(y)-\alpha}^k\ud y
  =\int_B(b(y)-\alpha)_+^k\ud y
    +\int_B(\alpha-b(y))_+^k\ud y 
   \lesssim \sum_{i=1}^2 \abs{\pair{1_{\tilde E_i}}{T_b^k 1_{E_i}}}.
\end{split}
\end{equation*}
\end{proof}

We present a variant of the result for rough homogeneous kernels. While the conclusion is essentially identical, the proof requires an additional iteration of the basic argument.

\begin{lemma}\label{lem:medianRough}
Let $k\in\{1,2,\ldots\}$, let $\Omega\in L^1(S^{d-1})\setminus\{0\}$ and $\displaystyle K(x)=\frac{\Omega(x/\abs{x})}{\abs{x}^d}$.
Let $\theta_0\in S^{d-1}$ be a Lebesgue point of $K$, where $K(\theta_0)=\Omega(\theta_0)\neq 0$. Let $T$ be an operator with kernel $K(x,y)=K(x-y)$.

Then there is a (large) constant $A$, depending only on the above data, such that every $b\in L^k_{\loc}(\R^d)$ satisfies
the following estimate for every ball $B$:
\begin{equation*}
  \inf_c\int_B\abs{b(y)-c}^k\ud y\lesssim\sum_{i=1}^4\abs{\pair{1_{\tilde E_i}}{T_b^k 1_{E_i}}}
\end{equation*}
for some subsets $E_i\subset B$ and $\tilde E_i\subset\tilde B:=B+Ar_B\theta_0$.
\end{lemma}

\begin{proof}
Given $B=B(y_0,r)$, let $x_0=y_0+Ar\theta_0$, where the large $A$ is yet to be chosen, and $\tilde B=B(x_0,r)$.

The basic observation is that, if $b(x)\leq\alpha$, then
\begin{equation*}
\begin{split}
  &\int_B (b(y)-\alpha)_+^k\ud y
  \leq \int_{B\cap\{b\geq\alpha\}} (b(y)-b(x))^k\ud y \\
  &= \int_{B\cap\{b\geq\alpha\}} (b(y)-b(x))^k\frac{(Ar)^d K(x_0-y_0)}{\Omega(\theta_0)}\ud y \\
  &= \frac{(-1)^k c_d\abs{B}}{\Omega(\theta_0)}A^d \int_{B\cap\{b\geq\alpha\}} (b(x)-b(y))^k [K(x-y)+K(x_0-y_0)-K(x-y)]\ud y \\
\end{split}
\end{equation*}
Hence, taking $\alpha$ as the median of $b$ on $\tilde B$, we have
\begin{equation*}
\begin{split}
  &\int_B (b(y)-\alpha)_+^k\ud y\leq 2\frac{\abs{\tilde B\cap\{b\leq\alpha\}}}{\abs{\tilde B}} \int_B (b(y)-\alpha)_+^k\ud y \\
  &\lesssim  A^d \abs{\pair{1_{\tilde B\cap\{b\leq\alpha\}}}{T_b^k 1_{B\cap\{b\geq\alpha\}}}}+ \\
  &\qquad +A^d\int_{\tilde B\cap\{b\leq\alpha\}}\int_{B\cap\{b\geq\alpha\}} (b(y)-b(x))^k\abs{K(x-y)-K(x_0-y_0)}\ud y\ud x.
\end{split}
\end{equation*}
Estimating
\begin{equation*}
  (b(y)-b(x))^k=(b(y)-\alpha+\alpha-b(x))^k
  \leq c_k(b(y)-\alpha)^k+c_k(\alpha-b(x))^k,
\end{equation*}
the double integral can be dominated by the sum of
\begin{equation*}
  \int_{B\cap\{b\geq\alpha\}}(b(y)-\alpha)^k\Big(\int_{\tilde B\cap\{b\leq\alpha\}}\abs{K(x-y)-K(x_0-y_0)}\ud x\Big)\ud y
\end{equation*}
and
\begin{equation*}
  \int_{\tilde B\cap\{b\leq\alpha\}}(\alpha-b(x))^k\Big(\int_{B\cap\{b\geq\alpha\}}\abs{K(x-y)-K(x_0-y_0)}\ud y\Big)\ud x.
\end{equation*}
Writing $x-y=x_0-y_0+(x-x_0)-(y-y_0)$, both inner integrals are seen to be bounded by
\begin{equation*}
\begin{split}
  \int_{B(0,2r)} &\abs{K(x_0-y_0+z)-K(x_0-y_0)}\ud z \\
  &=\int_{B(0,2r)}\abs{K(Ar\theta_0+z)-K(Ar\theta_0)}\ud z \\
  &=(Ar)^{-d}\int_{B(0,2r)}\abs{K(\theta_0+\frac{z}{Ar})-K(\theta_0)}\ud z \\
  &=\int_{B(0,2/A)}\abs{K(\theta_0+u)-K(\theta_0)}\ud u
  =\eps_A A^{-d},
\end{split}
\end{equation*}
where $\eps_A\to 0$ as $A\to\infty$, by the assumption that $\theta_0$ is a Lebesgue point of $K$.

Substituting back, and observing in particular the cancellation of the factors $A^d$ and $A^{-d}$ in the double integral, we have proved that
\begin{equation*}
\begin{split}
  \int_B (b(y)-\alpha)_+^k\ud y 
  &\leq c_A\abs{\pair{1_{\tilde B\cap\{b\leq\alpha\}}}{T_b^k 1_{B\cap\{b\geq\alpha\}}}} \\
  &\qquad+c\eps_A \int_B(b(y)-\alpha)_+^k\ud y+c\eps_A\int_{\tilde B}(\alpha-b(x))_+^k\ud x,
\end{split}
\end{equation*}
and hence
\begin{equation*}
\begin{split}
  (1-c\eps_A)\int_B (b(y)-\alpha)_+^k\ud y 
  &\leq c_A\abs{\pair{1_{\tilde B\cap\{b\leq\alpha\}}}{T_b^k 1_{B\cap\{b\geq\alpha\}}}} \\
  &\qquad+c\eps_A\int_{\tilde B}(\alpha-b(x))_+^k\ud x.
\end{split}
\end{equation*}
Replacing $(b,\alpha)$ by $(-b,-\alpha)$, we also have
\begin{equation*}
\begin{split}
  (1-c\eps_A)\int_B (\alpha-b(y))_+^k\ud y 
  &\leq c_A\abs{\pair{1_{\tilde B\cap\{b\geq\alpha\}}}{T_b^k 1_{B\cap\{b\leq\alpha\}}}} \\
  &\qquad+c\eps_A\int_{\tilde B}(b(x)-\alpha)_+^k\ud x,
\end{split}
\end{equation*}
and adding the two estimates,
\begin{equation*}
  (1-c\eps_A)\int_B \abs{b(y)-\alpha}^k\ud y 
  \leq c_A\sum_{i=1}^2 \abs{\pair{1_{\tilde E_i}}{T_b^k 1_{E_i}}}
  +c\eps_A\int_{\tilde B}\abs{b(x)-\alpha}^k\ud x.
\end{equation*}
where $E_i\subset B$ and $\tilde E_i\subset\tilde B$ for $i=1,2$.
Recall that $\alpha$ was the median of $b$ on $\tilde B$, but since this choice of $\alpha$ is a quasi-minimiser for the integral on the right, we also deduce the more symmetric version
\begin{equation}\label{eq:moreSymmVers}
  \inf_{\alpha\in\R}\int_B \abs{b(y)-\alpha}^k\ud y 
  \leq c\sum_{i=1}^2 \abs{\pair{1_{\tilde E_i}}{T_b^k 1_{E_i}}}+\frac{1}{2}\inf_{\alpha\in\R}\int_{\tilde B}\abs{b(x)-\alpha}^k\ud x,
\end{equation}
where we have also fixed an $A$ so that $c\eps_A/(1-c\eps_A)\leq 1/2$.

We now apply the same argument to the adjoint
\begin{equation*}
  (T_b^k)^*=(-1)^k(T^*)_b^k.
\end{equation*}
We note that the kernel $K^*$ of $T^*$ is related to the kernel $K$ of $T$ given by $K^*(x,y)=K(y,x)$, and hence it is also a homogeneous kernel with symbol $\Omega^*(\theta)=\Omega(-\theta)$. In particular, the point $-\theta_0$ plays the same role for $T^*$ as $\theta_0$ plays for $T$, and thus the ball $B=\tilde B-Ar\theta_0$ plays the same role for $\tilde B$ and $T^*$ as $\tilde B$ plays for $B$ and $T$. 

 This means that the analogue of \eqref{eq:moreSymmVers} in the adjoint case reads as
\begin{equation*}
  \inf_{\alpha\in\R}\int_{\tilde B} \abs{b(x)-\alpha}^k\ud x 
  \leq c\sum_{i=3}^4 \abs{\pair{1_{\tilde E_i}}{T_b^k 1_{E_i}}}+\frac12\inf_{\alpha\in\R}\int_{B}\abs{b(y)-\alpha}^k\ud y,
\end{equation*}
where again $E_i\subset B$ and $\tilde E_i\subset\tilde B$ for $i=3,4$.
Using \eqref{eq:moreSymmVers} and its adjoint version above consecutively, we have
\begin{equation*}
\begin{split}
  &\inf_{\alpha\in\R}\int_B \abs{b(y)-\alpha}^k\ud y 
  \leq c\sum_{i=1}^2 \abs{\pair{1_{\tilde E_i}}{T_b^k 1_{E_i}}}
  +\frac{1}{2}\inf_{\alpha\in\R}\int_{\tilde B}\abs{b(x)-\alpha}^k\ud x \\
  &\leq c\sum_{i=1}^2 \abs{\pair{1_{\tilde E_i}}{T_b^k 1_{E_i}}} 
  +\frac{1}{2}\Big(c\sum_{i=3}^4 \abs{\pair{1_{\tilde E_i}}{T_b^k 1_{E_i}}}
  +\frac{1}{2}\inf_{\alpha\in\R}\int_B \abs{b(y)-\alpha}^k\ud y \Big) \\
  &\leq c\sum_{i=1}^4 \abs{\pair{1_{\tilde E_i}}{T_b^k 1_{E_i}}}
  +\frac{1}{4}\inf_{\alpha\in\R}\int_B \abs{b(y)-\alpha}^k\ud y,
\end{split}
\end{equation*}
which implies that
\begin{equation*}
  \inf_{\alpha\in\R}\int_B \abs{b(y)-\alpha}^k\ud y 
  \lesssim \sum_{i=1}^4 \abs{\pair{1_{\tilde E_i}}{T_b^k 1_{E_i}}}
\end{equation*}  
as claimed.
\end{proof}

\subsection{The lower bound for higher commutators}

We restate and then prove the ``only if'' parts of Theorem \ref{thm:higher} in the two theorems below, dealing with the cases $p\leq q$ and $p>q$, in analogy with Theorems \ref{thm:pleq1} and \ref{thm:offDiag}.

\begin{theorem}\label{thm:pleqk}
Let $K$ be a non-degenerate Calder\'on--Zygmund kernel, let $k\in\{1,2,\ldots\}$ and $b\in L^k_{\loc}(\R^d)$.
Let further
\begin{equation*}
   1<p\leq q<\infty,\qquad \alpha:=\frac{d}{k}\Big(\frac{1}{p}-\frac{1}{q}\Big)\geq 0,
\end{equation*}
and suppose that $T_b^k$ satisfies the following weak form of $L^p\to L^q$ boundedness:
\begin{equation}\label{eq:veryWeakLpLqk}
\begin{split}
  \abs{\pair{T_b^k f}{g}}
  &=\Babs{\iint (b(x)-b(y))^kK(x,y)f(y)g(x)\ud y\ud x} \\
  &\leq \Theta\cdot \Norm{f}{\infty}\abs{B}^{1/p}\cdot\Norm{g}{\infty}\abs{\tilde B}^{1/q'},
\end{split}
\end{equation}
whenever $f\in L^\infty(B)$, $g\in L^\infty(\tilde B)$ for any two balls of equal radius $r$ and distance $\dist(B,\tilde B)\gtrsim r$.
Then
\begin{itemize}
  \item if $\alpha=0$, equivalently $p=q$, we have  $b\in\BMO(\R^d)$, and
  $\Norm{b}{\BMO}^k\lesssim \Theta;$ 
  \item if $\alpha\in(0,1]$, we have $b\in\dot C^{0,\alpha}(\R^d)$, and
  $\Norm{b}{\dot C^{0,\alpha}}^k\lesssim \Theta;$ 
  \item if $\alpha>0$, the function $b$ is constant, so in fact $T_b^k=0$.
\end{itemize}
\end{theorem}

\begin{proof}
Consider a ball $B$ of radius $r$.
From Lemma \ref{lem:quposBalls} and Remark \ref{rem:quposBalls} or Lemma \ref{lem:medianRough} (depending whether $K$ is a two-variable or rough homogeneous kernel), it is immediate that
\begin{equation}\label{eq:iterBasic}
  \Big(\fint_{B}\abs{b-\ave{b}_B}\Big)^k
  \leq\fint_{B}\abs{b-\ave{b}_{B}}^k
  \lesssim \sum_{i=1}^4 \frac{ \abs{ \pair{1_{\tilde E_j} }{ T_b^k 1_{E_j} } } }{\abs{B}},
\end{equation}
for some subsets $E_j\subset B$, $\tilde E_j\subset\tilde B$, where $\tilde B$ is a ball of the same radius $r$ and $\dist(B,\tilde B)\gtrsim r$.
By assumption \eqref{eq:veryWeakLpLqk}, it follows that
\begin{equation*}
  \abs{ \pair{1_{\tilde E_j} }{ T_b^k 1_{E_j} }}
  \leq \Theta\abs{B}^{1/p}\abs{B}^{1/q'}
  =\Theta\abs{B}^{1/p-1/q+1}=\Theta\abs{B}r^{\alpha k}
\end{equation*}
and hence
\begin{equation*}
  \fint_{B}\abs{b-\ave{b}_{B}}\lesssim \Theta^{1/k}r^\alpha.
\end{equation*}
From this the rest follows as in the proof of Theorem \ref{thm:pleq1}.
\end{proof}

\begin{theorem}\label{thm:offDiagk}
Let $K$ be a non-degenerate Calder\'on--Zygmund kernel, let $k\in\{1,2,\ldots\}$, and $b\in L^k_{\loc}(\R^d;\R)$. Let
\begin{equation*}
  1<q<p<\infty,\qquad r=\frac{pq}{p-q}\in(1,\infty),
\end{equation*}
and suppose that $T_b^k$ satisfies the following weak form of $L^p\to L^q$ boundedness:
\begin{equation}\label{eq:quiteWeakLpLqk}
  \sum_{i=1}^N\abs{\pair{T_b^k f_i}{g_i}}
  \leq \Theta\BNorm{\sum_{i=1}^N\Norm{f_i}{\infty} 1_{Q_i}}{p}
    \BNorm{\sum_{i=1}^N\Norm{g_i}{\infty} 1_{\tilde Q_i}}{q'},
\end{equation}
whenever, for each $i=1,\ldots,N$, we have $f_i\in L^\infty(Q_i)$ and $g_i\in L^\infty(\tilde Q_i)$ for cubes $Q_i$ and $\tilde Q_i$ such that $\dist(Q_i,\tilde Q_i)\gtrsim\diam(Q_i)=\diam(\tilde Q_i)$.

Then $b=a+c$ for some $a\in L^{rk}(\R^d)$ and some constant $c\in\C$, where
$  \Norm{a}{rk}\lesssim\Theta.$
\end{theorem}

\begin{proof}
Let us fix some (large) cube $Q_0\subset\R^d$. We apply Lemma \ref{lem:Lerner} to find that
\begin{equation*}
  1_{Q_0}\abs{b-\ave{b}_{Q_0}}
  \lesssim \sum_{Q\in\mathcal Q} 1_Q\fint_Q\abs{b-\ave{b}_Q}
  \leq \sum_{j=1}^\infty 1_{Q_j}\Big(\fint_{Q_j}\abs{b-\ave{b}_{Q_j}}^k\Big)^{1/k},
\end{equation*}
where we also introduced an enumeration of the sparse collection $\mathcal Q$ of dyadic subcubes of $Q_0$ given by Lemma \ref{lem:Lerner}.

By Lemma \ref{lem:quposBalls} and Remark \ref{rem:quposBalls}, or Lemma \ref{lem:medianRough} in the rough homogeneous case,
we have
\begin{equation*}
  \Big(\fint_{Q_j}\abs{b-\ave{b}_{Q_j}}^k\Big)^{1/k}
  \lesssim \sum_{i=1}^4 \Big(\frac{ \abs{ \pair{1_{\tilde E^i_j} }{ T_b^k 1_{E^i_j} } } }{\abs{Q_j}}\Big)^{1/k}. 
\end{equation*}
for some subsets $E^i_j\subset Q_j$ and $\tilde E^i_j\subset \tilde Q_j$, where the cube $\tilde Q_j$ satisfies $\dist(Q_j,\tilde Q_j)\gtrsim\diam(Q_j)=\diam(\tilde Q_j)$.
Hence
\begin{equation*}
  \Norm{b-\ave{b}_{Q_0}}{L^{kr}(Q_0)}
  \lesssim\sup\Big\{\sum_{i=1}^4\sum_{j=1}^\infty\Big(\frac{ \abs{ \pair{1_{\tilde E^i_j} }{ T_b^k 1_{E^i_j} } } }{\abs{Q_j}}\Big)^{1/k} \int_{Q_j}\phi:\Norm{\phi}{L^{(kr)'}}\leq 1\Big\}.
\end{equation*}
It is enough to give a uniform bound for the finite sums
\begin{equation*}
\begin{split}
  &\sum_{j=1}^N\Big(\frac{ \abs{ \pair{1_{\tilde E^i_j} }{ T_b^k 1_{E^i_j} } } }{\abs{Q_j}}\Big)^{1/k}\abs{Q_j} \fint_{Q_j}\phi \\
  &\leq \Big(\sum_{j=1}^N\Big(\frac{ \abs{ \pair{1_{\tilde E^i_j} }{ T_b^k 1_{E^i_j} } } }{\abs{Q_j}}\Big)^{r}\abs{Q_j}\Big)^{1/(kr)}
    \Big(\sum_{j=1}^N\abs{Q_j} \big[\fint_{Q_j}\phi\big]^{(kr)'}\Big)^{1/(kr)'}.
\end{split}
\end{equation*}
Using the sparseness of $\mathcal Q\supset\{Q_j\}_{j=1}^N$, we can bound the second factor by
\begin{equation*}
  \sum_{j=1}^N\abs{Q_j} \big[\fint_{Q_j}\phi\big]^{(kr)'}
  \lesssim\sum_{j=1}^N\abs{E(Q_j)}\inf_{z\in Q_j}M\phi(z)^{(kr)'}
  \leq\int (M\phi)^{(kr)'}\lesssim\int \phi^{(kr)'}\leq 1.
\end{equation*}
We dualise the first factor with $\sum_{j=1}^N \lambda_j^{r'}\abs{Q_j}\leq 1$ to end up considering
\begin{equation}\label{eq:pgeqMedStep}
\begin{split}
  &\sum_{j=1}^N \frac{  \abs{\pair{1_{\tilde E^i_j} }{ T_b^k 1_{E^i_j} } } }{\abs{Q_j}}\lambda_j \abs{Q_j} 
   =\sum_{j=1}^N  \abs{\pair{ \lambda_j^{r'/q} 1_{\tilde E^i_j} }{ T_b^k (\lambda_j^{r'/p} 1_{E^i_j} ) }}    \\
  &\leq \Theta\BNorm{\sum_{j=1}^N \lambda_j^{r'/q'}1_{\tilde Q_j}}{q'} \BNorm{\sum_{j=1}^N \lambda_j^{r'/p}1_{Q_j}}{p},
\end{split}
\end{equation}
where we used the assumption \eqref{eq:quiteWeakLpLqk} in the last step.

By Lemma \ref{lem:Doob}, we have, using the disjoint major subsets $E(Q_j)\subset Q_j$,
\begin{equation*}
   \BNorm{\sum_{j=1}^N \lambda_j^{r'/p}1_{Q_j}}{p}
   \lesssim \BNorm{\sum_{j=1}^N \lambda_j^{r'/p}1_{E(Q_j)}}{p}
   =\Big(\sum_{j=1}^N \lambda_j^{r'}\abs{E(Q_j)}\Big)^{1/p}\leq 1.
\end{equation*}
For the first factor on the right of \eqref{eq:pgeqMedStep}, we obtain a similar bound by starting with
\begin{equation*}
  \BNorm{\sum_{j=1}^N \lambda_j^{r'/q'}1_{\tilde Q_j}}{q'} 
  \lesssim \BNorm{\sum_{j=1}^N \lambda_j^{r'/q'}1_{Q_j}}{q'},
\end{equation*}
which also follows from Lemma \ref{lem:Doob}, and then finishing as before.

We have now proved that
\begin{equation*}
  \Norm{b-\ave{b}_{Q_0}}{L^{kr}(Q_0)}\lesssim \Theta^{1/k}
\end{equation*}
for any cube $Q_0\subset\R^d$. This shows in particular that $b\in L^{kr}_{\loc}(\R^d)$, and we conclude by Lemma \ref{lem:loc2glob}.
\end{proof}

\subsection{Two-weight norm inequalities of Bloom type}\label{ss:two-weight}

We finally discuss the boundedness of commutators between weighted $L^p$ spaces with weights from the Muckenhoupt class
\begin{equation*}
  A_p(\R^d)=\Big\{w\in L^1_{\loc}(\R^d):w>0\text{ a.e.},\  [w]_{A_p}:=\sup_B \fint_B w\Big(\fint_B w^{-\frac{1}{p-1}}\Big)^{p-1}<\infty\Big\},
\end{equation*}
where the supremum is over all balls $B\subset\R^d$. We consider $p\in(1,\infty)$ fixed throughout this discussion, and denote by $w':=w^{-\frac{1}{p-1}}$ the {\em dual weight}.
 One checks that $w\in A_p$ if and only if $w'\in A_{p'}$. The space $L^{p'}(w')$ is the dual of $L^p(w)$ with respect to the {\em unweighted} duality $\pair{f}{g}=\int fg$. We will identify a weight and its induced measure, using notation like $w(Q):=\int_Q w$.

We will be concerned with the boundedness of
\begin{equation*}
  T_b^k:L^p(\mu)\to L^p(\lambda),\qquad\mu,\lambda\in A_p(\R^d),
\end{equation*}
i.e., we allow two different weights on the domain and the target space, but (in contrast to the rest of the paper) we restrict the Lebesgue exponents to $p=q\in(1,\infty)$. This fits with the line of investigation that was started by Bloom \cite{Bloom:85} and that has been recently revived by Holmes, Lacey, and Wick \cite{HLW}, followed by several others as we shortly recall. Here we complete the following picture:

\begin{theorem}\label{thm:Bloom}
Let $T$ be a non-degenerate Calder\'on--Zygmund operator, let $k\in\{1,2,\ldots\}$ and $b\in L^k_{\loc}(\R^d;\R)$.
Let further $p\in(1,\infty)$ and $\lambda,\mu\in A_p(\R^d)$.
Then $T_b^k$ defines a bounded operator $L^p(\mu)\to L^p(\lambda)$ if and only if
\begin{equation*}
  \Norm{b}{\BMO(\nu^{1/k})}:=\sup_B\frac{1}{\nu^{1/k}(B)}\int_B\abs{b-\ave{b}_B}<\infty,\qquad\nu:=(\mu/\lambda)^{1/p}.
\end{equation*}
\end{theorem}

The first version of Theorem \ref{thm:Bloom}, when $k=d=1$ and $T$ is the Hilbert transform, is due to Bloom \cite{Bloom:85}.  
Still for first order commutators $(k=1)$ but in arbitrary dimension $d\geq 1$, Holmes, Lacey, and Wick \cite{HLW} proved the ``if'' part of Theorem \ref{thm:Bloom} for all standard Calder\'on--Zygmund operators, and the ``only if'' part assuming the boundedness of each of the $d$ Riesz transforms $R_i$, $i=1,\ldots,d$, thus extending the exact scope (in terms of operators) of the classical Coifman--Rochberg--Weiss theorem  \cite{CRW} to the two-weight setting. The first two-weight result for iterated commutators was achieved in the ``if'' direction by Holmes and Wick \cite{HolWick} (with a simplified proof in \cite{Hyt:HolWick}): they obtained the boundedness of $T_b^k:L^p(\mu)\to L^p(\lambda)$ for any $k\geq 1$ under the stronger condition that $b\in\BMO(\nu)\cap\BMO(\R^d)\subset\BMO(\nu^{1/k})$. (For the inclusion, which in general is strict, see \cite[Lemma 4.7]{LORR}.) Finally, Lerner, Ombrosi, and Rivera-R{\'{\i}}os \cite{LORR}  obtained Theorem \ref{thm:Bloom} almost as stated: they identified the correct BMO space with the weight $\nu^{1/k}$ depending on the order $k$ of the commutator, and they proved the ``if'' part of Theorem \ref{thm:Bloom} for all standard Calder\'on--Zygmund operators and the ``only if'' part for all homogeneous Calder\'on--Zygmund operators with the fairly general local positivity assumption discussed in Section \ref{ss:intro-nec}. For us, it remains to prove this ``only if'' part assuming non-degeneracy only, and more precisely we prove:

\begin{theorem}\label{thm:Bloom2}
Let $K$ be a non-degenerate Calder\'on--Zygmund kernel, let $k\in\{1,2,\ldots\}$ and $b\in L^k_{\loc}(\R^d;\R)$.
Let $p\in(1,\infty)$, let $\lambda,\mu\in A_p(\R^d)$,
and suppose that $T_b^k$ satisfies the following weak form of $L^p(\mu)\to L^p(\lambda)$ boundedness:
\begin{equation}\label{eq:veryWeakBloom}
\begin{split}
  \abs{\pair{T_b^k f}{g}}
  &=\Babs{\iint (b(x)-b(y))^kK(x,y)f(y)g(x)\ud y\ud x} \\
  &\leq \Theta\cdot \Norm{f}{\infty}\mu(B)^{1/p}\cdot\Norm{g}{\infty}\lambda'(\tilde B)^{1/p'},
\end{split}
\end{equation}
whenever $f\in L^\infty(B)$, $g\in L^\infty(\tilde B)$ for any two balls of equal radius $r$ and distance $\dist(B,\tilde B)\gtrsim r$.
Then $b\in\BMO(\nu^{1/k})$, where $\nu=(\mu/\lambda)^{1/p}$, and more precisely
\begin{equation*}
   \Norm{b}{\BMO(\nu^{1/k})}\lesssim \Theta^{1/k}.
\end{equation*}
\end{theorem}

Let us first observe that \eqref{eq:veryWeakBloom} is indeed a weak form of the boundedness of $T_b^k:L^p(\mu)\to L^p(\lambda)$: if this boundedness holds, then
\begin{equation*}
\begin{split}
  \abs{\pair{T_b^k f}{g}}
  &\leq\Norm{T_b^k f}{L^p(\lambda)}\Norm{g}{L^{p'}(\lambda')}
  \leq\Norm{T_b^k}{L^p(\mu)\to L^p(\lambda)}\Norm{f}{L^p(\mu)}\Norm{g}{L^{p'}(\lambda')} \\
  &\leq \Norm{T_b^k}{L^p(\mu)\to L^p(\lambda)}\cdot \Norm{f}{\infty}\mu(B)^{1/p}\cdot\Norm{g}{\infty}\lambda'(\tilde B)^{1/p'},
\end{split}
\end{equation*}
and thus $\Theta\leq \Norm{T_b^k}{L^p(\mu)\to L^p(\lambda)}$.

Turning to the proof of Theorem \ref{thm:Bloom2}, we need a simple lemma, which is the only place where the $A_p$ condition is used.

\begin{lemma}\label{lem:basicAp}
Let $\lambda,\mu\in A_p$.
If $B,\tilde B$ are balls of equal radius $r$ with $\dist(B,\tilde B)\lesssim r$, then
\begin{equation*}
  \mu(B)^{1/p}\lambda'(\tilde B)^{1/p'}\lesssim \ave{\nu^{1/k}}_B^k\cdot \abs{B}
\end{equation*}
for all $k=1,2,\ldots$, where $\nu=(\mu/\lambda)^{1/p}$.
\end{lemma}

\begin{proof}
We recall that all $A_p$ weights, and then also $\lambda'\in A_{p'}$, are doubling. Hence $\lambda'(\tilde B)\lesssim\lambda'(B)$.
We then use the $A_p$ property of both $\mu$ and $\nu$ directly via the definition (together with some basic algebra involving $p$ and $p'$) to see that
\begin{equation*}
\begin{split}
  \frac{\mu(B)^{1/p}\lambda'(\tilde B)^{1/p'}}{\abs{B}}
  \lesssim\ave{\mu}_B^{1/p}\ave{\lambda'}_B^{1/p'}
  \lesssim\frac{1}{\ave{\mu'}_B^{1/p'}\ave{\lambda}_B^{1/p}}
  \leq\frac{1}{\ave{(\mu')^{1/p'}\lambda^{1/p}}_B}
  =\frac{1}{\ave{\nu^{-1}}_B}
\end{split}
\end{equation*}
Finally,
\begin{equation*}
  1=\ave{\nu^{-\frac{1}{k+1}}\nu^{\frac{1}{k}\cdot\frac{k}{k+1}}}_B
  \leq\ave{\nu^{-1}}^{\frac{1}{k+1}}_B\ave{\nu^{\frac{1}{k}}}_B^{\frac{k}{k+1}},
\end{equation*}
and hence $\ave{\nu^{-1}}_B^{-1}\leq\ave{\nu^{1/k}}^k_B$.
\end{proof}

\begin{proof}[Proof of Theorem \ref{thm:Bloom2}]
As in the proof of the unweighted version in Theorem \ref{thm:pleqk}, we have (just copying \eqref{eq:iterBasic} from the said proof)
\begin{equation*}
  \Big(\fint_{B}\abs{b-\ave{b}_B}\Big)^k
  \leq\fint_{B}\abs{b-\ave{b}_{B}}^k
  \lesssim \sum_{i=1}^4 \frac{ \abs{ \pair{1_{\tilde E_j} }{ T_b^k 1_{E_j} } } }{\abs{B}},
\end{equation*}
for some subsets $E_j\subset B$, $\tilde E_j\subset\tilde B$, where $\tilde B$ is a ball of the same radius $r$ and $\dist(B,\tilde B)\eqsim r$.

By assumption \eqref{eq:veryWeakBloom} and Lemma \ref{lem:basicAp}, we have
\begin{equation*}
  \frac{ \abs{ \pair{1_{\tilde E_j} }{ T_b^k 1_{E_j} } } }{\abs{B}}
  \leq \Theta\frac{ \mu(B)^{1/p}\lambda'(\tilde B)^{1/p'} }{\abs{B}}
  \lesssim \Theta\ave{\nu^{1/k}}_B^k.
\end{equation*}
Hence
\begin{equation*}
   \Big(\fint_{B}\abs{b-\ave{b}_B}\Big)^k\lesssim \Theta\ave{\nu^{1/k}}_B^k,
\end{equation*}
which simplifies to $\Norm{b}{\BMO(\nu^{1/k})}\lesssim \Theta^{1/k}$.
\end{proof}

\subsubsection*{Acknowledgements}
I would like to thank Sauli Lindberg for bringing the problem about the Jacobian operator and its connection to commutators to my attention, and for pointing out some oversights in an earlier version of the manuscript. I would also like to thank Riikka Korte for discussions on the theme of the paper, and the anonymous referee for constructive suggestions on the presentation.

\subsubsection*{Declarations of interest:} None.


\begin{thebibliography}{10}

\bibitem{AFMR}
H.~Aimar, L.~Forzani, F.~J. Mart\'in-Reyes, 
On weighted inequalities for singular integrals. Proc. Amer. Math. Soc. 125 (1997), no. 7, 2057--2064.

\bibitem{AHLMO}
E.~Airta, T.~Hyt\"onen, K.~Li, H.~Martikainen, T.~Oikari,
Off-diagonal estimates for bi-commutators.
\newblock \href {http://arxiv.org/abs/2005.03548} {\path{arXiv:2005.03548}}.


\bibitem{Bloom:85}
S.~Bloom, {A commutator theorem and
  weighted {BMO}}, Trans. Amer. Math. Soc. 292~(1) (1985) 103--122.
\newblock \href {https://doi.org/10.2307/2000172} {\path{doi:10.2307/2000172}}.

\bibitem{Chaffee:16}
L.~Chaffee, 
Characterizations of bounded mean oscillation through commutators of bilinear singular integral operators.
Proc. Roy. Soc. Edinburgh Sect. A 146 (2016), no. 6, 1159--1166.

\bibitem{ChaffeeCruz}
L.~Chaffee, D.~Cruz-Uribe, 
Necessary conditions for the boundedness of linear and bilinear commutators on Banach function spaces.
Math. Inequal. Appl. 21 (2018), no. 1, 1--16.

\bibitem{CLMS}
R.~Coifman, P.-L. Lions, Y.~Meyer, S.~Semmes, Compensated compactness and
  {H}ardy spaces, J. Math. Pures Appl. (9) 72~(3) (1993) 247--286.

\bibitem{CRW}
R.~R. Coifman, R.~Rochberg, G.~Weiss, Factorization theorems for {H}ardy spaces
  in several variables, Ann. of Math. (2) 103~(3) (1976) 611--635.

\bibitem{DM}
B.~Dacorogna, J.~Moser, {On
  a partial differential equation involving the {J}acobian determinant}, Ann.
  Inst. H. Poincar\'e Anal. Non Lin\'eaire 7~(1) (1990) 1--26.
\newblock \href {https://doi.org/10.1016/S0294-1449(16)30307-9}
  {\path{doi:10.1016/S0294-1449(16)30307-9}}.


\bibitem{DHKY}
G.~Dafni, T.~Hyt\"{o}nen, R.~Korte, H.~Yue,
  {The space {$JN_p$}:
  nontriviality and duality}, J. Funct. Anal. 275~(3) (2018) 577--603.
\newblock \href {https://doi.org/10.1016/j.jfa.2018.05.007}
  {\path{doi:10.1016/j.jfa.2018.05.007}}.


\bibitem{DLLW}
X.~T. Duong, H.-Q. Li, J.~Li, B.~D. Wick,
  {Lower bound of {R}iesz
  transform kernels and commutator theorems on stratified nilpotent {L}ie
  groups}, J. Math. Pures Appl. (9) 124 (2019) 273--299.
\newblock \href {https://doi.org/10.1016/j.matpur.2018.06.012}
  {\path{doi:10.1016/j.matpur.2018.06.012}}.

\bibitem{FL:Acta}
S.~H. Ferguson, M.~T. Lacey, 
A characterization of product BMO by commutators. 
Acta Math. 189 (2002), no. 2, 143--160.

\bibitem{GCRF}
J.~Garc\'ia-Cuerva, J.~L. Rubio de Francia, 
Weighted norm inequalities and related topics. 
North-Holland Mathematics Studies, 116. Notas de Matem\'atica, 104.
North-Holland Publishing Co., Amsterdam, 1985.

\bibitem{GLW}
W.~{Guo}, J.~{Lian}, H.~{Wu}, {The unified theory for the necessity of bounded
  commutators and applications}, 
  J. Geom. Anal. 30 (2020), no. 4, 3995--4035. 

\bibitem{HHL}
T.~S. H\"anninen, T.~P. Hyt\"onen, K.~Li,
  {Two-weight {$L^p$}-{$L^q$}
  bounds for positive dyadic operators: unified approach to {$p\leq q$} and
  {$p>q$}}, Potential Anal. 45~(3) (2016) 579--608.


\bibitem{HLW}
I.~Holmes, M.~T. Lacey, B.~D. Wick,
 {Commutators in the
  two-weight setting}, Math. Ann. 367~(1-2) (2017) 51--80.
\newblock \href {https://doi.org/10.1007/s00208-016-1378-1}
  {\path{doi:10.1007/s00208-016-1378-1}}.


\bibitem{HolWick}
I.~Holmes, B.~D. Wick, Two weight inequalities for iterated commutators with
  {C}alder\'{o}n-{Z}ygmund operators, J. Operator Theory 79~(1) (2018) 33--54.

\bibitem{Hyt:HolWick}
T.~P. Hyt\"onen, {The
  {H}olmes-{W}ick theorem on two-weight bounds for higher order commutators
  revisited}, Arch. Math. (Basel) 107~(4) (2016) 389--395.
\newblock \href {https://doi.org/10.1007/s00013-016-0956-5}
  {\path{doi:10.1007/s00013-016-0956-5}}.

\bibitem{Hyt:Ricci}
T.~P. Hyt\"onen, Of commutators and {J}acobians (2019).
\newblock \href {http://arxiv.org/abs/1905.00814} {\path{arXiv:1905.00814}}.

\bibitem{Iwaniec:Escorial}
T.~Iwaniec, {Nonlinear commutators and
  {J}acobians}, J. Fourier Anal. Appl. 3 (1997) 775--796.

\bibitem{Janson:1978}
S.~Janson, {Mean oscillation and
  commutators of singular integral operators}, Ark. Mat. 16~(2) (1978)
  263--270.

\bibitem{JN:61}
F.~John, L.~Nirenberg, On functions of bounded mean oscillation, Comm. Pure
  Appl. Math. 14 (1961) 415--426.

\bibitem{KRW}
K.~Koumatos, F.~Rindler, E.~Wiedemann,
  {Differential inclusions and {Y}oung
  measures involving prescribed {J}acobians}, SIAM J. Math. Anal. 47~(2) (2015)
  1169--1195.
\newblock \href {https://doi.org/10.1137/140968860}
  {\path{doi:10.1137/140968860}}.

\bibitem{Lacey:error}
M.~T. Lacey, S.~Petermichl, J.~C. Pipher, B.~D. Wick, 
Notification of error: multiparameter Riesz commutators.
Amer. J. Math. 143 (2021), no. 2, 333--334.

\bibitem{LSU:positive}
M.~T. Lacey, E.~T. Sawyer, I.~Uriarte-Tuero, Two weight inequalities for
  discrete positive operators (2010).
\newblock \href {http://arxiv.org/abs/0911.3437} {\path{arXiv:0911.3437}}.

\bibitem{Lerner:formula}
A.~K. Lerner, {A pointwise estimate
  for the local sharp maximal function with applications to singular
  integrals}, Bull. Lond. Math. Soc. 42~(5) (2010) 843--856.
\newblock \href {https://doi.org/10.1112/blms/bdq042}
  {\path{doi:10.1112/blms/bdq042}}.


\bibitem{LORR}
A.~K. Lerner, S.~Ombrosi, I.~P. Rivera-R\'{\i}os,
  {Commutators of singular integrals
  revisited}, Bull. Lond. Math. Soc. 51~(1) (2019) 107--119.
\newblock \href {https://doi.org/10.1112/blms.12216}
  {\path{doi:10.1112/blms.12216}}.

\bibitem{LiawTreil}
C.~Liaw, S.~Treil,
Regularizations of general singular integral operators.
Rev. Mat. Iberoam. 29 (2013), no. 1, 53--74.

\bibitem{Lindberg:2017}
S.~Lindberg, \href{https://doi.org/10.1007/s00205-017-1087-2}{On the {H}ardy
  space theory of compensated compactness quantities}, Arch. Ration. Mech.
  Anal. 224~(2) (2017) 709--742.

\bibitem{Oikari:awf}
T.~Oikari,
Approximate weak factorizations and bilinear commutators,
\newblock \href {http://arxiv.org/abs/2102.13535} {\path{arXiv:2102.13535}}.

\bibitem{PTW}
M.~Paluszy\'nski, M.~Taibleson, G.~Weiss, Characterization of {L}ipschitz
  spaces via the commutator operator of {C}oifman, {R}ochberg, and {W}eiss,
  Rev. Un. Mat. Argentina 37~(1-2) (1991) 142--144 (1992), X Latin American
  School of Mathematics (Spanish) (Tanti, 1991).

\bibitem{PauP}
J.~Pau, A.~Per\"{a}l\"{a}, {A
  {T}oeplitz-type operator on {H}ardy spaces in the unit ball}, Trans. Amer.
  Math. Soc. 373~(5) (2020) 3031--3062.
\newblock \href {https://doi.org/10.1090/tran/8053}
  {\path{doi:10.1090/tran/8053}}.


\bibitem{Sohr:NSE}
H.~Sohr, {The {N}avier-{S}tokes
  equations. An elementary functional analytic approach}, Birkh\"auser Advanced
  Texts: Basler Lehrb\"ucher., Birkh\"auser Verlag, Basel, 2001.
\newblock \href {https://doi.org/10.1007/978-3-0348-8255-2}
  {\path{doi:10.1007/978-3-0348-8255-2}}.

\bibitem{Stein:book}
E.~M. Stein, 
Harmonic analysis: real-variable methods, orthogonality, and oscillatory integrals. 
With the assistance of Timothy S. Murphy. Princeton Mathematical Series, 43.
Monographs in Harmonic Analysis, III. Princeton University Press, Princeton, NJ, 1993. 


\bibitem{Tanaka:Wolff}
H.~Tanaka, {A characterization
  of two-weight trace inequalities for positive dyadic operators in the upper
  triangle case}, Potential Anal. 41~(2) (2014) 487--499.

\bibitem{Uchiyama:1978}
A.~Uchiyama, {On the compactness
  of operators of {H}ankel type}, T\^ohoku Math. J. (2) 30~(1) (1978) 163--171.


\bibitem{Ye:1994}
D.~Ye, {Prescribing the
  {J}acobian determinant in {S}obolev spaces}, Ann. Inst. H. Poincar\'e Anal.
  Non Lin\'eaire 11~(3) (1994) 275--296.
\newblock \href {https://doi.org/10.1016/S0294-1449(16)30185-8}
  {\path{doi:10.1016/S0294-1449(16)30185-8}}.





\end{thebibliography}
\end{document}